\documentclass{amsart}
\usepackage{amsmath, mathrsfs,amsthm, amssymb,color,fullpage,amsrefs,mathtools,marvosym,enumerate}
\usepackage[utf8]{inputenc}
\usepackage[normalem]{ulem}
\usepackage{hyperref}
\usepackage{tikz}
\usetikzlibrary{decorations.pathreplacing}
 \def\beq{\begin{eqnarray}}
\def\eeq{\end{eqnarray}}
\newcommand{\nn}{\nonumber}

\newtheorem{thm}{Theorem}[section]
\newtheorem*{thm*}{Theorem}
\newtheorem{prop}[thm]{Proposition}
\newtheorem*{prop*}{Proposition}
\newtheorem{cor}[thm]{Corollary}
\newtheorem{rem}[thm]{Remark}
\newtheorem{lemma}[thm]{Lemma}

\newtheorem{exa}[thm]{Example}

\newtheorem*{question*}{Question}

\newtheorem{nota}[thm]{Notation}
\newtheorem{defn}[thm]{Definition}
\newtheorem{remark}[thm]{Remark}

\newtheorem{conv}[thm]{Convention}
\newtheorem{notaconv}[thm]{Notational Conventions}
\newtheorem{conj}[thm]{Conjecture}

\definecolor{pink}{rgb}{1,0,1}   
\newcommand{\pink}[1]{{\color{pink}{#1}}}

    \newcommand{\R}{\mathbb R}
    \newcommand{\pa}{\partial}
    
    \newcommand{\cD}{\mathscr{D}}

          \newcommand{\bell}{\pmb{\ell}}
         \newcommand{\balpha}{\pmb{\alpha}}

    \newcommand{\cH}{\mathcal{H}}

    \newcommand{\Z}{\mathbb Z}
\newcommand{\Om}{\Omega}        

         \def\abs{\operatorname{ab}}  
          \def\pmca{\pmb{C}_{\abs}}  
          
    \newcommand{\bxi}{\pmb{\xi}}
   
     \def\omred{\Om^{\operatorname{red}}}

\newcommand{\A}{\mathcal{A}}

\newcommand{\al}{\alpha}
     \newcommand{\Lc}{\mathcal{L}}
     \newcommand{\lmin}{\ell_{\min}}
         \newcommand{\pc}{\mathcal{P}}
      
\title[]{The Steklov spectrum of convex polygonal domains II: investigating spectral determination}

\author[E. Dryden]{Emily B. Dryden}
\author[C. Gordon]{Carolyn Gordon} 
\author[J. Moreno]{Javier Moreno}
\author[J. Rowlett]{Julie Rowlett}
\author[C. Villegas Blas]{Carlos Villegas-Blas}

\address{Emily Dryden, Department of Mathematics,  Bucknell University,  Lewisburg, PA 17837 USA} 
\urladdr{\href{http://www.unix.bucknell.edu/~ed012/}{http://www.unix.bucknell.edu/~ed012/}}
\email{\href{mailto:}{emily.dryden@bucknell.edu}}

\address{Carolyn Gordon, Department of Mathematics, Dartmouth College, Hanover, NH 03755 USA} 
\urladdr{\href{https://math.dartmouth.edu/~gordon/}{https://math.dartmouth.edu/~gordon/}}
\email{\href{mailto:}{carolyn.s.gordon@dartmouth.edu}}

\address{Javier Moreno, Department of Mathematics,  Universidad de Los Andes,  111711, Bogotá, Colombia}
\email{\href{mailto:}{jd.morenop@uniandes.edu.co}}
\address{Julie Rowlett,  Mathematical Sciences, Chalmers University,  412 96, Gothenburg, Sweden} 
\urladdr{\href{http://www.math.chalmers.se/~rowlett}{http://www.math.chalmers.se/~rowlett}}
\email{\href{mailto:julie.rowlett@chalmers.se}{julie.rowlett@chalmers.se}}

\address{Carlos Villegas-Blas, Instituto de Matemáticas
Universidad Nacional Autónoma de México, Ciudad Universitaria, 04510, México, CDMX}
\urladdr{\href{https://www.matem.unam.mx/fsd/villegas}{https://www.matem.unam.mx/fsd/villegas}}
\email{\href{mailto:carlos.villegas@im.unam.mx}{carlos.villegas@im.unam.mx}}

\begin{document}
\maketitle 

\begin{abstract}
The extent to which the geometry of an object is determined by some associated spectral data is a longstanding problem.  We investigate this problem in the context of the Steklov spectrum, focusing on convex polygons.  We prove that almost all triangles are uniquely determined by their Steklov spectra within the class of all triangles; further results depending on the types of angles in the triangles are given.  We examine three special classes of convex quadrilaterals--rectangles, parallelograms, and kites--and obtain results ranging from unique spectral determination to determination up to three possibilities.  For regular $n$-gons, we are again able to prove spectral determination within certain classes of polygons.  More generally, we investigate the extent to which the Steklov spectrum distinguishes convex polygons from simply-connected domains with smooth boundary; that is, does the Steklov spectrum detect corners?  We prove that triangles and quadrilaterals are spectrally distinguished from such smoothly bounded domains; moreover, we show that having the same Steklov spectrum as such a domain imposes substantial restrictions on the edge lengths of higher-order $n$-gons.  Throughout, our main tool is the characteristic polynomial developed in \cite{lpps19,klpps21}. 
\end{abstract}

\section{Introduction} \label{s:intro} 

Inverse spectral problems on simply-connected bounded planar domains ask the extent to which spectral data determine the geometry of the domain.  For example, in the case of the Dirichlet Laplacian, the spectrum determines the area, the perimeter, and the total boundary curvature.   Moreover, this spectrum distinguishes  domains with corners from those with smooth boundary \cite{RL2015, NRSI}, although it's not known whether it recognizes the number of corners.  Within the class of domains with smooth boundary, the Dirichlet Laplacian distinguishes ellipses of small eccentricity \cite{HZ2022}. In the case of polygons, the Dirichlet Laplacian distinguishes triangles within the class of triangles \cite{Durso, Grieser-Maronna}, regular polygons among all convex piecewise smooth planar domains \cite{EGS}, and non-obtuse trapezoids within the class of such trapezoids  \cite{HLRD}.  It's not known whether the Dirichlet Laplacian distinguishes among all quadrilaterals; even in seemingly simple geometric contexts, inverse spectral problems can resist solution.

We focus here on the Steklov eigenvalue problem on simply connected planar domains $\Omega$ with piecewise $C^1$ boundary, and especially on convex polygons.   The Steklov eigenvalue problem consists of finding all real numbers $\sigma$ for which a non-trivial function $u$ satisfies
\[ \Delta u = 0 \textrm{ in } \Omega, \quad \frac{\pa u}{\pa n} = \sigma u \, \, \textrm{ on } \pa \Omega,  
\]
where $\Delta$ is the Laplacian and 
 $\frac{\pa}{\pa n}$ is the exterior normal derivative.
The spectrum of the Steklov problem is discrete with 
\[ 0 = \sigma_1(\Omega) < \sigma_2 (\Omega) \leq \cdots \leq \sigma_m (\Omega) \leq \cdots \nearrow + \infty.\]
Equivalently, the Steklov eigenvalues are those of the Dirichlet-to-Neumann map, $\cD_\Omega$, 
\[ \left . \cD_\Omega : H^{1/2} (\pa \Omega) \to H^{-1/2} (\pa \Omega), \quad \cD_\Omega f = \frac{ \pa \cH_\Omega f}{\pa n} \right|_{\pa \Omega}.\]
 Here $H^s$ denotes the Sobolev space $W^{s,2}$, and $\cH_\Omega f$ is the harmonic extension of $f$ to $\Omega$.  For compact Riemannian manifolds with smooth boundary, the Dirichlet-to-Neumann operator is an elliptic pseudodifferential operator.    However, for manifolds with only piecewise smooth or less regular boundary, the Dirichlet-to-Neumann operator fails to be pseudodifferential.   
Nonetheless, Weyl asymptotics have been obtained for all compact Riemannian surfaces with Lipschitz boundary \cite[Theorem 1.1]{klp_23} and for bounded Euclidean domains of any dimension with Lipschitz boundary \cite[Theorem 1.2]{rozenblum23}.
Consequently, as in the case of the Dirichlet Laplacian, Weyl asymptotics show that the perimeter is a Steklov spectral invariant; however, they say nothing about the area.  Area is not expected to be a Steklov spectral invariant, although no counterexamples are known in the case of planar domains.\footnote{The Dirichlet-to-Neumann operator, and thus the Steklov spectrum of a compact Riemannian surface, are invariant under conformal changes of metrics with conformal factor one on the boundary.  This yields ``trivial'' examples of surfaces of different area with the same Steklov spectrum.}
For much more extensive background on the Steklov problem, we refer interested readers to survey articles on the Steklov problem by Girouard \& Polterovich \cite{gir_pol} and by Colbois et. al \cite{pre_survey}.  For historical background and physical implications, we refer to Kuznetsov et. al \cite{legacy}. 

Two bounded planar domains (or, more generally, compact Riemannian surfaces) $\Om$ and $\Om'$ with smooth boundary have the same Steklov asymptotics to infinite order, i.e., $\sigma_j(\Om)-\sigma_j(\Om')=O(j^{-\infty})$, if and only if they have exactly the same number and lengths of boundary components \cite{GPPS}.  However, the situation is much more complicated when one allows corners.  The article of  Levitin, Parnovski, Polterovich and Sher \cite{lpps19} and  subsequent article \cite{klpps21}, joint also with Krymski, yield powerful results in the case of curvilinear polygons with all interior angles in $(0,\pi)$. They associate to each such $\Om$ a trigonometric polynomial $P_\Om$, referred to as the \emph{characteristic polynomial} of $\Om$.  The polynomial depends only on the edge lengths and angles of $\Om$.  In the former paper, they show that the roots of $P_\Om$ yield the (lower-order) Steklov spectral asymptotics of $\Om$.   In the latter, they show that the characteristic polynomial is a Steklov spectral invariant.   By applying this invariant, they show that the Steklov spectrum determines the edge lengths for generic curvilinear $n$-gons with angles in $(0,\pi)$, and it moreover determines the angles up to countably many explicit possibilities.  The genericity conditions, referred to as \emph{admissibility}, consist of an incommensurability condition on the edge lengths together with exclusion of angles of the form $\frac{\pi}{2m+1}$ with $m\in \Z^+$.  

The current article and its predecessor \cite{steklov1} address the inverse Steklov problem for convex Euclidean polygons.   Our primary tool is the characteristic polynomial described above.   In \cite{steklov1}, we showed that if $\Om$ is a convex $n$-gon satisfying the generic conditions of admissibility, then there are at most finitely many convex $n$-gons that are Steklov isospectral to $\Om$.   Moreover, if all the angles of $\Om$ are obtuse, then $\Om$ is uniquely determined by its Steklov spectrum among all convex $n$-gons.  We also obtained lower bounds, depending only on the minimal interior angle and the perimeter, on the Steklov eigenvalues of any convex $n$-gon.   Consequently, there is a uniform lower bound on the angles of any collection of mutually Steklov isospectral convex $n$-gons.    This, together with the characteristic polynomial, enabled us to obtain Steklov spectral finiteness results for $n$-gons under weaker genericity conditions than admissibility.

The current article focuses on the following:

\begin{itemize}
\item Many interesting classes of polygons, such as isosceles triangles, rectangles, kites, regular polygons, etc., fail the genericity conditions. In such cases, we use the characteristic polynomial along with Euclidean geometric considerations to obtain Steklov finiteness results and, more often, explicit bounds on the size of Steklov isospectral sets.  We also obtain some Steklov spectral uniqueness results.
\item In contrast to the case of the Dirichlet spectrum, it is not known whether the Steklov spectrum distinguishes simply-connected planar domains from those with corners.   We investigate this question in the context of convex polygons versus smoothly-bounded, simply-connected domains.   
\end{itemize}

We note that while there are many examples known of Steklov isospectral non-isometric manifolds in all dimensions greater than or equal to two (e.g., \cite{GHW}), the question of existence of Steklov isospectral non-isometric plane domains remains open. 

\subsection{Overview of our results}
In these introductory comments, we will frame our positive results in terms of Steklov isospectrality.   However, in the proofs of all the affirmative results establishing Steklov spectral uniqueness or finiteness, we do not use the full strength of the Steklov spectrum.  We only use the characteristic polynomial, which as noted above is a Steklov spectral invariant.  The actual statements of the results in the remainder of the paper will be stated with the weaker hypothesis.    In all cases, uniqueness or finiteness will mean up to congruence.

\subsubsection{Triangles}  A sampling of our results on triangles:
\begin{itemize}
    \item Any Steklov isospectral set of triangles is finite.
    Moreover, almost all triangles are uniquely determined by their Steklov spectra within the class of all triangles.   (See Theorem~\ref{th:triangles123} for a more complete statement.)
    \item Within the class of all non-obtuse triangles, each triangle is uniquely determined by its Steklov spectrum.    Moreover, right triangles and almost all acute triangles can be distinguished from obtuse triangles by their Steklov spectra.   (See Theorem~\ref{th:non_obtuse} and Proposition~\ref{prop.acute among all}.)
    \item  Isosceles triangles are mutually distinguishable by their Steklov spectra. (See Theorem~\ref{th:iso_tri}.) However, as noted above, our results do not use the full strength of the Steklov spectrum but only the characteristic polynomial; this polynomial does not always distinguish isosceles triangles from scalene triangles. (See Remark~\ref{rem:iso_triangles}.)  It is still open whether the Steklov spectrum distinguishes them.   
\end{itemize}

\subsubsection{Convex quadrilaterals}  A sampling of our results on three special classes of convex quadrilaterals:

\begin{itemize}
\item Rectangles: Within the class of convex quadrilaterals and triangles, each rectangle is uniquely determined by its Steklov spectrum.  (See Theorem 4.1.)
Note that \cite[Cor. 1.8]{cuboid} shows that each rectangle is uniquely determined within the class of all rectangles.

\item Parallelograms: Depending on the angles that appear in a parallelogram, either the edge lengths or angles of the parallelogram are uniquely determined by the Steklov spectrum. (See Theorem 4.2 for more precise statements.)
\item Kites: Within the class of kites, almost all kites are determined up to three possibilities by their Steklov spectra, and a substantial collection of kites are uniquely determined. (See \S \ref{ss:kites} for more precise statements.)
\end{itemize}

\subsubsection{Regular polygons}  A sampling of our results on regular polygons:
\begin{itemize}
\item Within the class of regular $n$-gons, each element is uniquely determined by its Steklov spectrum. (See Theorem 5.3.)
\item Regular $n$-gons are distinguished by their Steklov spectrum from all $m$-gons, $m>n$, with all angles greater than $\frac{\pi}{3}$, and from all non-equilateral $n$-gons with all angles greater than $\frac{\pi}{3}$. (These results follow from Example 5.2.)  
\item We have additional results for small values of $n$.  As one example, no convex quadrilateral can have the same Steklov spectrum as an equilateral triangle. (See Theorem 5.5 and Remark 5.6.)  
\end{itemize}

\subsubsection{Convex polygons and smooth domains} In exploring whether the Steklov spectrum distinguishes convex polygons from simply-connected domains with smooth boundary, we obtained results including the following:
\begin{itemize}
\item The Steklov spectrum distinguishes triangles and quadrilaterals from smoothly bounded simply-connected plane domains. (See Theorem 6.1.)

\item For higher-order $n$-gons, we show that the possibility of Steklov isospectrality to a smoothly-bounded simply-connected plane domain entails substantial restrictions on the edge lengths. 
\end{itemize}

\subsection{Plan of the paper}
Section~\ref{s:prelim} contains the necessary background, focusing primarily on the characteristic polynomial and results of \cite{klpps21}.   Sections~\ref{ss:triangles}, \ref{s:quad}, and \ref{s:regular} address triangles, quadrilaterals, and regular polygons, respectively. In Section~\ref{s:smooth}  we explore the question of distinguishing convex polygons from smoothly-bounded, simply-connected plane domains via the characteristic polynomial. 

\section*{Acknowledgements} This work was initiated at the BIRS-CMO workshop 22w5149.  We sincerely thank the organizers as well as all sponsors of the workshop. We are grateful to David Sher, Alexandre Girouard and Iosif Polterovich for inspiring and insightful discussions and correspondence. C. Villegas-Blas was partially supported by 
projects CONACYT Ciencia B\'asica  CB-2016-283531-F-0363 and UNAM-PAPIIT-IN 116323 and 115126. 

\section{Preliminaries} \label{s:prelim} 
Although we focus on convex polygons, we will use results developed in \cite{klpps21} for simply-connected curvilinear $n$-gons in $\R^2$; we assume these $n$-gons have piecewise smooth edges and all angles at the vertices lie in the interval $(0,\pi)$.   We follow the same labeling convention for the edge lengths and interior angles at the vertices as  \cite{klpps21} and as described in Notational Conventions 2.1 of \cite{steklov1}.

\begin{notaconv}\cite[Nota.\! Conv. 2.1]{steklov1}\label{nota:bell and bal} We use $\ell_1,\dots, \ell_n$ to denote edge lengths and $\alpha_1,\dots,\alpha_n$ to denote the interior angles at the vertices.   We will usually abuse notation and use the same notation $\ell_j$, respectively  $\alpha_j$, to denote the $j$th edge, respectively vertex.  We always number the edges and vertices cyclically with vertex $\alpha_j$ occurring between edges $\ell_j$ and $\ell_{j+1}$ (see Figure~\ref{fig:tri_con}); $\ell_{n+1}$ is understood to be $\ell_1$.   The data associated with a curvilinear $n$-gon $\Om$ consists of its vectors of edge lengths and angles
$$\bell=(\ell_1,\dots,\ell_n)\mbox{\,\,\,and\,\,\,}\balpha=(\alpha_1,\dots, \alpha_n).$$
The cyclic labeling is unique only up to $2n$ possible permutations, corresponding to a choice of orientation of $\partial\Om$ and a choice of initial edge.    
\end{notaconv}

 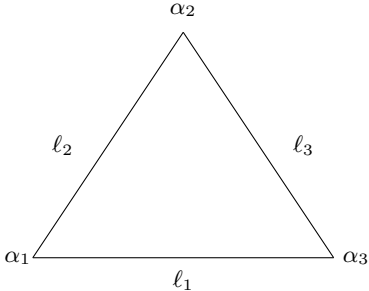
\begin{figure}[h] \centering 
\begin{tikzpicture}
\draw   (-2,0) --  (2,0);
\draw  (-2, 0) -- (0, 3); 
\draw (0, 3) -- (2, 0); 
\node at (-2.2,0) {\small $\alpha_1$}; 
\node at (-1.6, 1.5) {\small $\ell_2$}; 
\node at (0, -0.3) {\small $\ell_1$}; 
\node at (2.3,0) {\small $\alpha_3$}; 
\node at (0, 3.3) {\small $\alpha_2$};
\node at (1.6, 1.5) {\small $\ell_3$}; 
\end{tikzpicture}
\caption{A triangle with angles and edges labeled as in \cite{klpps21}.}
\label{fig:tri_con}
\end{figure}  

Every curvilinear $n$-gon $\Om$ has an associated characteristic polynomial $P_\Om$, first introduced in \cite[Equation (2.20)]{lpps19} (see also \cite{klpps21}).  This will be our main tool for obtaining inverse spectral results.

\begin{defn}\label{def: char poly} Let $\Om$ be a curvilinear $n$-gon with data $\bell$ and $\balpha$ as in Notational Conventions~\ref{nota:bell and bal}.
Define 
\beq\label{eq:c(alpha)} c(\alpha_j)=\cos\left(\frac{\pi^2}{2\alpha_j}\right) \ \ \ \text{and} \ \ \ s(\alpha_j)=\sin\left(\frac{\pi^2}{2\alpha_j}\right).\eeq 
For $\bxi \in \{\pm 1\}^n$, we set 
\[
 a_{\bxi}=\prod_{\{j:\xi_j\neq \xi_{j+1}\}}\,c(\alpha_j),
 \]
where $a_{\bxi}$ equals $1$ if the product is over the empty set.   The subscripts in $\bxi$ are cyclically ordered, so $\xi_{n+1}$ is understood to be $\xi_1$.
 {\rm (}In the definition of $a_{\bxi}$, the product is either empty or contains an even number of factors, since there is always an even number of sign changes as one moves cyclically through the entries of $\bxi$ in order to return to the starting value.{\rm )} Define an equivalence relation on $\{\pm 1\}^n$ by \[\bxi\sim\bxi ' \iff \bxi=\pm\bxi'.\] 
 Denote the equivalence class of $\bxi$ by $[\bxi]$, and define $a_{[\bxi]} :=a_{\bxi}$.
The \em characteristic polynomial \em $P_\Om: \R\to \R$
is a trigonometric polynomial given by
\[ P_{\Omega} (t) := \sum_{[\bxi]\in \{\pm1\}^n/\sim}\, a_{[\bxi]} \cos(|\bxi \cdot \bell| t) - \prod_{j=1} ^n s (\alpha_j).\]

The definition of curvilinear $n$-gon is understood to include the case $n=0$, with a 0-gon being a simply-connected smoothly bounded plane domain.  The characteristic polynomial in this case is given by 
 \beq P_\Om(t)=\cos(\ell t) -1. \label{eq:cp_smooth_domain} \eeq
 where $\ell$ is the perimeter. 
\end{defn}  

\begin{conv}\label{conv:perim} Observe that the maximal cosine frequency occurring in $P_\Om$ is $\ell_1+\dots + \ell_n$ and the term $\cos(|\ell_1+\dots +\ell_n| t)$, corresponding to $\bxi = \pm (1,1,\dots, 1),$ necessarily has coefficient equal to one and frequency equal to the perimeter.  Thus the perimeter is an invariant of the characteristic polynomial.  Moreover, rescaling all the edge lengths by a constant $c$ results in rescaling all the cosine frequencies by the factor $c$ and does not otherwise affect the characteristic polynomial.    Thus we will always assume that all (possibly curvilinear) $n$-gons under consideration have perimeter one.
\end{conv}

The characteristic polynomial $P_\Om$ depends only on the geometry of $\Om$, more specifically only on $\balpha(\Om)$ and $\bell(\Om)$.  Details about the relationship between the roots of $P_\Om$ and the Steklov eigenvalues can be found in \cite{lpps19} and \cite{klpps21}; for our present purposes, the key result is the following.  

\begin{thm}\cite[Theorem 1.16]{klpps21}
    \label{poly is spec invar}
The characteristic polynomial $P_{\Omega}$ of a curvilinear $n$-gon $\Om$ with all angles in $(0,\pi)$ can be constructed algorithmically from the Steklov spectrum of $\Om$.   In particular, the characteristic polynomial is a Steklov spectral invariant of $\Om$.\end{thm}

Theorem \ref{poly is spec invar} shows that if two curvilinear polygons share a common Steklov spectrum, then they also share the same characteristic polynomial. A more precise statement is that two curvilinear polygons $\Om$ and $\Om'$ have the same characteristic polynomial if and only if their Steklov eigenvalues satisfy $\sigma_j(\Om)-\sigma_j(\Om') =o(1)$ as $j\to \infty$.    (See \cite[Theorem 1.13 and Remark 1.15]{klpps21}.) Numerically, it appears that the characteristic polynomial does not completely determine the Steklov spectrum; see, for instance, Examples 3.5 and 3.6 as well as Remark 3.9 in \cite{klpps21}. To explore the geometric features encoded in the characteristic polynomial, we recall the following two definitions that respectively match Definitions 2.8 and 2.9 of \cite{steklov1}.

\begin{defn} \label{def:exceptional}
We will say that an angle is \emph{rational} if it is a rational multiple of $\pi$.  Among the rational angles, those of the form $\frac{\pi}{k}$, where $k\in \Z$, will play an especially important role in what follows.     Angles of this form will be called \em odd\em, respectively \em even\em, angles 
if $k$ is an odd, respectively even, positive integer.   \emph{(}These angles are referred to as ``special,'' respectively ``exceptional,'' in \cite{klpps21}.\emph{)}  Observe that an angle $\alpha$ is odd if and only if $c(\alpha)=0$, 
while even angles $\alpha=\frac{\pi}{2m}$ satisfy $c(\alpha)=(-1)^m$.
Following \cite{klpps21}, we will refer to $(-1)^m$ as the \emph{parity} of the even angle $\frac{\pi}{2m}$.  Similarly, we will refer to $(-1)^j$ as the parity of the odd angle $\frac{\pi}{2j+1}$. 
\end{defn} 

\begin{defn}\cite[Definition 1.8]{klpps21} \label{def:generic} A curvilinear $n$-gon with all interior angles in $(0,\pi)$ is said to be \em admissible \em if the following two conditions hold:  (1) the side lengths $\ell_1, \ldots, \ell_n$ are incommensurable over $\{-1, 0, +1\}$ (that is, no non-trivial linear combination of $\ell_1, \ldots, \ell_n$ with coefficients taken from $\{-1, 0, 1\}$ vanishes);  and (2) none of the interior angles $\alpha_1, \ldots, \alpha_n$ are odd    (see Definition~\ref{def:exceptional}). Observe that 0-gons, i.e., simply-connected smoothly bounded domains, always satisfy the admissibility conditions.  Admissible $1$-gons have a single vertex with interior angle that is not odd, and necessarily have a curved edge.
\end{defn} 

The following result shows that the characteristic polynomial distinguishes between admissible and non-admissible curvilinear $n$-gons.

\begin{prop}\label{prop:vertices}\cite[p. 22]{klpps21}
For $n\geq 1$, a curvilinear $n$-gon $\Omega$ with all interior angles in $(0, \pi)$ is admissible if and only if its characteristic polynomial $P_\Om$ contains exactly $2^{n-1}$ linearly independent terms of the form $a\cos(ct)$ with $c\neq 0$.  Moreover, within the class of all admissible curvilinear polygons, the characteristic polynomial determines the number of vertices. 
\end{prop}

\begin{remark}
From private communication, we know that the authors of \cite{klpps21}   interpreted the notion of ``curvilinear $n$-gons'' to include the case $n=0$.  In particular, $0$-gons can be included in the final statement of Proposition \ref{prop:vertices}. Indeed, when comparing the case $n=0$ and $n=1$, note that the characteristic polynomials of 1-gons and 0-gons both have only one cosine frequency; however, the constant term in the characteristic polynomial distinguishes admissible 1-gons from 0-gons. 
\end{remark}

As a straightforward consequence of this proposition (see Cor. 2.11 of \cite{steklov1}), we deduce that 
the characteristic polynomial distinguishes admissible curvilinear $n$-gons from all non-admissible curvilinear polygons that have at most $n$ vertices and have all interior angles in $(0,\pi)$.  However, an admissible curvilinear $n$-gon may have the same characteristic polynomial as a non-admissible curvilinear polygon with more than $n$ vertices; see Lemma~\ref{om vs omred}.   

The following result of \cite{klpps21} shows that one can recover considerable geometric information from the characteristic polynomial of an admissible curvilinear $n$-gon.

\begin{thm}\cite[Theorem 1.17]{klpps21} \label{th:ss_ngons}
We use the notation and terminology of~\ref{nota:bell and bal}, \ref{def: char poly}, \ref{def:exceptional}, and \ref{def:generic}.   
Suppose that $\Omega$ and $\Omega'$ are admissible curvilinear $n$-gons that have the same characteristic polynomial.  Then
\begin{enumerate} 
\item[(a)] $\Omega$ and $\Omega'$ have the same number of even angles.
\item[(b)] If they have no even angles, then the boundary orientations and cyclical labeling of the edges and vertices can be chosen so that 
\[ \bell(\Omega)=\bell(\Omega') \]
and
\[ (c(\al_1), \dots, c(\al_n))=\pm (c(\al_1'),\dots, c(\al_n'))\]
for some choice of $\pm$.   
\item[(c)] If $\Om$ and $\Om'$ have one or more even angles, then we have 
\[(\ell_1,\dots, \ell_n)=(\ell'_{\sigma(1)},  \dots \ell'_{\sigma(n)})\]
and 
\[(|c(\al_1)|,\dots, |c(\al_n)|)= (|c(\al'_{\sigma(1)})|,\dots, |c(\al'_{\sigma(n)})|)\]
where $\sigma$ is one of a few explicitly given permutations.
\end{enumerate}
\end{thm}  

Their result in the presence of even angles is actually considerably stronger than what we have stated in part (c) of the theorem above.   For a more complete statement, see \cite[Theorem 1.17]{klpps21}, also recalled in \cite[Theorem 2.13]{steklov1}.

\begin{remark}[\textbf{A practical guide to the characteristic polynomial}] \label{ss:cpguide}  Assume that two curvilinear $n$-gons with interior angles in $(0, \pi)$ have the same characteristic polynomial.  In each polynomial, collect the cosine terms that have the same frequency. By the linear independence of cosines with different frequencies, these frequencies must be the same for both $n$-gons.  Odd angles delete cosine terms from the characteristic polynomial, whereas even angles delete the product of sines.  Since a convex $n$-gon has angle sum $(n-2)\pi$ we deduce that, unless it is the equilateral triangle, it has at most two odd angles.  Moreover, with the exception of rectangles, a convex $n$-gon can have  at most three even angles. There are two ways that information can be lost if the edge lengths fail to satisfy the condition of incommensurability over $\{-1,0,1\}$:
\begin{itemize}
\item If there is a commensurability relationship with all coefficients in $\{-1,1\}$, then $\bxi\cdot \bell=0$ for some $\bxi$, and thus the term $a_{[\bxi]}\cos(|\bxi\cdot\bell|t)$ in the characteristic polynomial is absorbed into the constant term.
\item Other commensurability relationships result in having two or more terms of the form $a_{[\bxi]}\cos(|\bxi\cdot\bell|t)$ with the same cosine frequency.    In this case, cancellation may occur.
\end{itemize}
\end{remark}

We note the following elementary properties of $|c(\alpha)|$:
\begin{lemma}\cite[Lemma 2.16]{steklov1}\label{lem: |c|}~
Define $|c|:(0,\pi)\to [0,1]$ by $|c|(\alpha):=|c(\alpha)|$ where $c(\alpha)=\cos\left(\frac{\pi^2}{2\alpha}\right)$ as in Definition~\ref{def: char poly}.  Then:
\begin{enumerate}

\item[(a)]  $|c|^{-1}(\{0\})$ consists of all odd angles $\frac{\pi}{2k+1}$, $k\in \Z^+$.
\item[(b)] $|c|^{-1}(\{1\})$ consists of all even angles $\frac{\pi}{2k}$, $k\in \Z^+$.
\item[(c)] $|c|$ maps each interval $(\frac{\pi}{m+1},\frac \pi m)$, $m\in \Z^+$, bijectively onto $(0,1)$.     In particular, the restriction of $|c|$ to the set of all obtuse angles is injective.
\item[(d)] For $s\in [0,1]$, the inverse image $|c|^{-1}(\{s\})$ is discrete and accumulates only at $0$.
\end{enumerate}
\end{lemma}

\section{Triangles} \label{ss:triangles} 

We will see that at most finitely many triangles can share the same characteristic polynomial and that having the same characteristic polynomial implies congruence for large classes of triangles.  We remind the reader that determination by the characteristic polynomial will always mean determination up to congruence, and that all triangles will be assumed to be normalized to have perimeter one.

The characteristic polynomial of a triangle $T$ of perimeter one with angles $\alpha_1,\alpha_2,\alpha_3$ and edge lengths $\ell_1,\ell_2,\ell_3$ can be expressed as
\begin{eqnarray}
\label{eq:triangle_charpoly}
P_{T}(t) &=& \cos (t)\\  \nonumber
&+& \,c(\alpha_1)c(\alpha_3)\cos(|1-2\ell_1|t) \,+ \,c(\alpha_1)c(\alpha_2)\cos(|1-2\ell_2|t)
 \,+\,c(\alpha_2)c(\alpha_3)\cos(|1-2\ell_3|t)\\  \nonumber
&-&s(\alpha_1)s(\alpha_2)s(\alpha_3).\nonumber
\end{eqnarray}
Here we are using, for example, that $\ell_1-\ell_2+\ell_3=1-2\ell_2$.  From $P_{T}$, we can recover information about the lengths and angles of $T$.

\begin{remark} \label{rem:oneodd_tri}  If a triangle has one odd angle, say $\alpha_1$, then by \eqref{eq:triangle_charpoly}, its characteristic polynomial determines $\ell_3$, and thus $\ell_1+\ell_2$. Using the coefficient $c(\alpha_2)c(\alpha_3)$ together with the term $-s(\alpha_1)s(\alpha_2)s(\alpha_3) = \pm s(\alpha_2)s(\alpha_3)$ and angle addition/subtraction identities for the cosine, one can also recover the unordered set $\{|c(\alpha_2)|, |c(\alpha_3)|\}$.  Alternatively, this follows immediately from \cite[Lemma 5.3]{steklov1}, restated as Lemma \ref{om vs omred} in \S 6.
\end{remark}

\begin{thm}  \label{th:triangles123} ~Let $T$ be a triangle.
\begin{enumerate}  
\item[(a)] The characteristic polynomial of $T$ determines the number of odd angles in $T$.
    \item[(b)]  If $T$ has no odd angles, then it is determined by its characteristic polynomial within the set of all triangles.  In particular, every admissible triangle is uniquely determined by its characteristic polynomial.
    \item[(c)] If $T$ has precisely one odd angle, then there are at most four triangles with the same characteristic polynomial.  Moreover, the characteristic polynomial together with the measure of any one angle uniquely determine $T$ within the set of all triangles.
        \item[(d)] If $T$ has two odd angles, the characteristic polynomial uniquely determines the third angle and determines $T$ up to finitely many possibilities within the set of all triangles.
    \item[(e)] If $T$ has three odd angles, then it is equilateral and is uniquely determined among all triangles by its characteristic polynomial. 
\end{enumerate}
Thus any Steklov isospectral set of triangles contains at most finitely many elements. 
\end{thm}

\begin{proof}
(a)  Consider the three terms on the second line of Equation~\eqref{eq:triangle_charpoly}.  The triangle inequality implies that the cosine frequences $1-2\ell_j$, $j=1, 2, 3$ are never zero.   If $T$ has no odd angles, then all the coefficients are non-zero and the three cosine frequencies are either all different (scalene case) or exactly two coincide (isosceles case).   Here we are using the symmetry of isosceles triangles, which guarantees that the coefficients of the two like terms can't cancel.   Thus in either case, the three terms yield at least two independent cosine functions.   If $T$ has exactly one odd angle, then two of the three terms vanish, leaving exactly one cosine function.   If $T$ has two or more odd angles, then all three terms vanish.   Thus the number of independent cosine functions in $P_T$ determines whether $T$ has zero, one, or more than one odd angle.   To distinguish between two versus three odd angles, it's enough to look at the constant term $s(\alpha_1)s(\alpha_2)s(\alpha_3)$.   This term has magnitude one if and only if  $T$ has three odd angles.  Hence the characteristic polynomial determines the number of odd angles.

(b)  Assume $T$ has no odd angles.  If $T$ is a scalene triangle, the distinct cosine frequencies $1-2\ell_j$, $j=1,2,3$ in the characteristic polynomial immediately give us the three edge lengths.   If $T$ is isosceles, we can read off the two distinct edge lengths.   Using the fact that the perimeter is one, we can determine which of the two lengths occurs twice.  Thus in either case, the characteristic polynomial determines all three edge lengths and thus determines $T$.   

(c)     Let $T$ have exactly one odd angle $\alpha_1$.  As observed in Remark \ref{rem:oneodd_tri}, $P_T(t)$ determines $\ell_3$ and $\ell_1+\ell_2$.    
Let $E$ be the ellipse given by all points $p$ such that the sum of the distance from $p$ to the foci $\alpha_2$ and $\alpha_3$ is equal to $\ell_1+\ell_2$.    Then $\alpha_1$ must lie on $E$ as illustrated in Figure \ref{fig:ellipse}.  This information together with knowledge of the angle at a single vertex is enough to determine the triangle.   More precisely, if we can either determine the odd angle $\alpha_1$ or if we can determine one of $\alpha_2$ or $\alpha_3$ (unlabeled as long as we know it corresponds to one of $\alpha_2$ or $\alpha_3$), then we know $T$.   This proves the second statement of (c).

Since the odd angle is at most $\frac{\pi}{3}$ and $T$ is not equilateral, at least one of the other two angles -- call it $\alpha$ -- is strictly greater than $\frac{\pi}{3}$.  By Remark \ref{rem:oneodd_tri},
$P_T(t)$ determines the set $\{|c(\alpha_2)|, |c(\alpha_3)|\}$, so we know $|c(\alpha)|$ up to two possibilities.   By Lemma~\ref{lem: |c|}, we then know $\alpha$ up to at most four possibilities.   The first statement of (c) thus follows from the second.

(d) and (e) The only triangle with three odd angles is the equilateral triangle, so (e) follows from (a).   Now assume $T$ has exactly two odd angles $\alpha_1$ and $\alpha_2$.   We first prove that $P_T(t)$ determines $\alpha_3$.  Since $T$ is not equilateral, we have $\alpha_1+\alpha_2<\frac{2\pi}{3}$, so $\alpha_3>\frac{\pi}{3}$.  Moreover, $\alpha_3$ is obtuse unless $T$ is the triangle with angles $\frac{\pi}{3}, \frac{\pi}{5}, \frac{7\pi}{15}$.  Since $|s(\alpha)|=1$ when $\alpha$ is odd, the constant term in the characteristic polynomial is $|s(\alpha_3)|$.  Thus $P_T(t)$ determines $|s(\alpha_3)|$; equivalently, it determines $|c(\alpha_3)|$.   By Lemma~\ref{lem: |c|}, $|c|$ is injective on the set of obtuse angles.   Also $(\frac{\pi}{3},\pi)\,\cap \, |c|^{-1}\left(|c|(\frac{7\pi}{15})\right) =\{\frac{7\pi}{15}, \frac{7\pi}{13}\}$.   One can easily check that no triangle with two odd angles can have $\alpha_3=\frac{7\pi}{13}$.    Thus within the set of all triangles with two odd angles, the third angle is determined by the characteristic polynomial.  To prove finiteness, observe that given $\alpha_3$, there can be at most finitely many pairs $(\alpha_1,\alpha_2)$ of odd angles with $\alpha_1+\alpha_2=\pi -\alpha_3$.    
\end{proof}

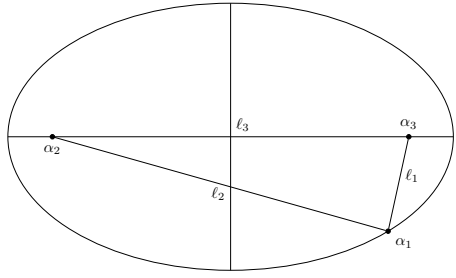
\begin{figure} \centering 
\resizebox{6cm}{!}{

\begin{tikzpicture}[dot/.style={draw,fill,circle,inner sep=1pt}]
  \def\a{5} 
  \def\b{3} 
  \def\angle{-45} 
  \draw (0,0) ellipse ({\a} and {\b});
  \draw (-\a,0) 
    -- (\a,0) ;
  \draw (0,-\b)
    -- (0,\b) ;
  \coordinate[label={above right:$\ell_3$}] (O) at (0,0);
  \node[dot,label={below:$\alpha_2$}] (F1) at ({-sqrt(\a*\a-\b*\b)},0) {};

  \node[dot,label={above:$\alpha_3$}] (F2) at ({+sqrt(\a*\a-\b*\b)},0) {};
  \node[dot,label={\angle:$\alpha_1$}] (X) at (\angle:{\a} and {\b}) {};
  \draw (F1) -- (X) (X) -- (F2);
    \coordinate[label={below left:$\ell_2$}] (L2) at (0,-1);
    \coordinate[label={below right:$\ell_1$}] (L1) at (3.8,-0.6);

\end{tikzpicture}}
\caption{Since the sum of the distances from the vertex $\alpha_1$ to the other two vertices is fixed and equal to $\ell_1 + \ell_2$, the vertex $\alpha_1$ must lie on the ellipse with foci at the vertices $\alpha_2$ and $\alpha_3$ as depicted here.} \label{fig:ellipse} 
\end{figure} 

If we restrict to the class of isosceles triangles, the characteristic polynomial suffices to uniquely determine each triangle.

\begin{thm} \label{th:iso_tri}
Within the class of isosceles triangles, each element is uniquely determined by its characteristic polynomial.
\end{thm}
\begin{proof}
Let $T$ be an isosceles triangle.  By Theorem \ref{th:triangles123}, the number of odd angles in $T$ is known; moreover, if there are three odd angles, then $T$ is the equilateral triangle of perimeter one. 

Assume that $T$ has precisely two odd angles.  Theorem~\ref{th:triangles123}(d) shows that the third angle is uniquely determined by the characteristic polynomial. Thus the remaining two equal angles are also known, and since we know that the triangle has perimeter one, T is determined.
 
If $T$ has precisely one odd angle, then it is the vertex angle at which the two equal sides meet. Denoting this angle by $\alpha_1$, Remark~\ref{rem:oneodd_tri} implies that $P_T(t)$ determines $\ell_1 + \ell_2$ and $\ell_3$.  But $\ell_1=\ell_2$, so all three edge lengths are determined and thus so is $T$.  

Finally, if there are no odd angles, Theorem \ref{th:triangles123} implies that $T$ is determined.
\end{proof}  

\begin{remark} \label{rem:iso_triangles} 
There are limited possibilities for a scalene triangle to have the same characteristic polynomial as an isosceles triangle.  An isosceles triangle that is not equilateral and has at least one odd angle must have precisely one or two odd angles, and all angles must be rational.  The characteristic polynomial is however not enough to distinguish isosceles triangles from non-isosceles.  For example, the triangle with angles $\frac{\pi}{3}, \frac{\pi}{15}, \frac{3\pi}{5}$ and the isosceles triangle with angles $\frac{\pi}{5}, \frac{\pi}{5}, \frac{3\pi}{5}$ have the same characteristic polynomial.  So to uniquely determine isosceles triangles among all triangles based on their Steklov spectrum, one would need to use more than the characteristic polynomial. This example also shows that the finiteness statement of Theorem \ref{th:triangles123}(d) cannot be improved to uniqueness.
\end{remark}

In the broad class of all non-obtuse triangles, each triangle is uniquely determined by its characteristic polynomial.

\begin{thm} \label{th:non_obtuse}
Within the class of all non-obtuse triangular domains, each domain is uniquely determined by its characteristic polynomial.
\end{thm} 
 
\begin{proof}
By Theorem~\ref{th:triangles123}, the characteristic polynomial $P_T(t)$ determines the number of odd angles in a triangular domain $T$; moreover, $T$ is uniquely determined if it contains no or three odd angles.  If a non-obtuse triangle $T$ contains exactly two odd angles, they are $\frac{\pi}{3}$ and $\frac{\pi}{5}$, and $T$ is uniquely determined.

It remains to show that if two non-obtuse triangles $T$ and $T'$ each have precisely one odd angle, then $T \cong T'$.  Denote the angles of $T$ and $T'$ by $\alpha_i$ and $\alpha'_i$, respectively. Let $\alpha_1$, respectively $\alpha_1'$, be the odd angle in $T$, respectively $T'$.   
We make three key observations:
\begin{enumerate}
\item If $\alpha_1=\alpha_1'$, then $T \cong T'$ by Theorem~\ref{th:triangles123}(c).  Thus we may assume $\alpha_1\neq\alpha_1'$.   
\item Both triangles must have at least one angle in $(\frac{\pi}{3},\frac{\pi}{2}]$.  We relabel the angles, so $\alpha_2, \alpha_2'\in \left(\frac{\pi}{3},\frac{\pi}{2}\right].$   

\item If at least one of $\alpha_3$ or $\alpha_3'$ also lies in $(\frac{\pi}{3},\frac{\pi}{2}]$, then $T \cong T'$.    Indeed, suppose $\alpha_3\in (\frac{\pi}{3},\frac{\pi}{2}]$.  By Remark \ref{rem:oneodd_tri}, 
$P_T(t)$ determines the set $\{|c(\alpha_2)|, |c(\alpha_3)|\}$.  By Lemma~\ref{lem: |c|}, $|c|$ maps $(\frac{\pi}{3},\frac{\pi}{2}]$ bijectively onto $(0,1]$, so $\alpha_2'$ equals one of $\alpha_2$ or $\alpha_3$.   We can again conclude by Theorem~\ref{th:triangles123}(c) that $T \cong T'$.  Hence we assume that $\alpha_3, \alpha_3'<\frac{\pi}{3}$.
\end{enumerate}  

We conclude that $\alpha_2+\alpha_3<\frac{\pi}{2}+\frac{\pi}{3}$ and hence the odd angle $\alpha_1$ is one of $\frac{\pi}{3}$ or $\frac{\pi}{5}$.  The same holds for $\alpha_1'$.  After possibly interchanging the roles of $T$ and $T'$ if needed, we have $\alpha_1=\frac{\pi}{5}$ and $\alpha_1'=\frac{\pi}{3}.$

Consider the angles of $T$. Since  
$$\frac{\pi}{3}>\alpha_3=\pi-\frac{\pi}{5}-\alpha_2 \geq \frac{3\pi}{10},$$
we have
$\alpha_3\in \left[ \frac{3\pi}{10}, \frac{\pi}{3}\right)$, and thus $\frac{\pi^2}{2\alpha_3} \in (\frac{3\pi}{2}, \frac{5\pi}{3}]$. Consequently, $c(\alpha_3) > 0$ and $s(\alpha_3) < 0$.
We also have $\alpha_1+\alpha_3< \frac{8\pi}{15}$, so 
$\alpha_2\in \left(\frac{7\pi}{15}, \frac{\pi}{2}\right]$, and thus $\frac{\pi^2}{2\alpha_2} \in [\pi, \frac{15\pi}{14})$.  Consequently, $c(\alpha_2) < 0$ and $s(\alpha_2) \leq 0$.  So $c(\alpha_3)c(\alpha_2) < 0$ and $s(\alpha_2) s(\alpha_3) \geq 0$. 

For $T'$, a priori we just have $\alpha'_1=\frac{\pi}{3}$, $\alpha'_2\in \left(\frac{\pi}{3}, \frac{\pi}{2}\right]$, and $\alpha'_3\in\left[\frac{\pi}{6}, \frac{\pi}{3}\right)$.  Thus $\frac{\pi^2}{2\alpha_2'} \in [\pi, \frac{3\pi}{2})$.  Hence $c(\alpha_2') < 0$ and $s(\alpha_2') \leq 0$.  By the equality of the characteristic polynomials for $T$ and $T'$, we see that $c(\alpha'_2)c(\alpha'_3)=c(\alpha_2)c(\alpha_3)<0$.  Since $c(\alpha_2') <0$, we must have $c(\alpha_3')>0.$  Thus $\frac{\pi^2}{2\alpha_3'}$ must lie in the right half plane.  Since the odd angles $\alpha_1$ and $\alpha_1'$ have opposite parity, we must have 
$$s(\alpha_2') s(\alpha_3')=-s(\alpha_2)s(\alpha_3)\leq 0.$$   
There are two mutually exclusive possibilities:
\begin{enumerate}
    \item  These products of sines are zero and both triangles have an even angle.
    \item Neither triangle has an even angle.  Since $s(\alpha_2') \leq 0$, we must have $s(\alpha_3') > 0$, so $\frac{\pi^2}{2\alpha_3'}$ must lie in the first quadrant and thus $\frac{\pi^2}{2\alpha_3'}\in \left(2\pi,\frac{5\pi}{2}\right).$   
\end{enumerate}

 We first consider the case that neither triangle has an even angle.  If $|c(\alpha_2)|=|c(\alpha'_2)|$, then $\alpha_2=\alpha_2'$ due to Lemma~\ref{lem: |c|}(c) applied to the interval $(\frac{\pi}{3},\frac{\pi}{2})$.  In this case, $T \cong T'$ by Theorem~\ref{th:triangles123}(c).  Otherwise, we must have that $|c(\alpha_2)|=|c(\alpha'_3)|$, so 
$\frac{\pi^2}{2\alpha_3'}\in \left(2\pi,\frac{29\pi}{14}\right).$
Since $c(\alpha_2')=-c(\alpha_3)$, we have 
$\frac{\pi^2}{2\alpha'_2}\in \left[\frac{4\pi}{3},\frac{3\pi}{2}\right).$
This yields 
$\alpha'_2\in \left(\frac{\pi}{3},\frac{3\pi}{8}\right]\,\,\mbox{and\,\,}\alpha'_3\in \left(\frac{7\pi}{29},\frac{\pi}{4}\right).$
Thus 
$$\alpha'_1+\alpha'_2+\alpha'_3 < \frac{\pi}{3}+\frac{3\pi}{8}+\frac{\pi}{4}=\frac{23\pi}{24}<\pi,$$
a contradiction.

Next suppose that both $T$ and $T'$ have an even angle.  Since $\alpha_1=\frac{\pi}{5}$ and $T$ is non-obtuse, the even angle in $T$ must be a right angle, so $T$ has angles $\alpha_1=\frac{\pi}{5}, \alpha_2=\frac{\pi}{2}, \alpha_3=\frac{3\pi}{10}$.     
By Theorem~\ref{th:triangles123}(c), we may assume that $T'$ is not a right triangle.   Since $\alpha_1'=\frac{\pi}{3}$ and $T'$ is an acute triangle, the even angle in $T'$ must be $\frac{\pi}{4}$, so $T'$ has angles $\alpha'_1=\frac{\pi}{3}, \alpha_2'=\frac{\pi}{4}, \alpha_3'=\frac{5\pi}{12}.$     
  We have $c(\alpha_2')=-c(\alpha_2)$. 
However, $|c(\alpha_3')|\neq |c(\alpha_3)|$, which contradicts the fact that $T$ and $T'$ have the same characteristic polynomial.  
\end{proof}

Theorem~\ref{th:non_obtuse} states that non-obtuse triangles are mutually distinguishable by their characteristic polynomials.   In Proposition \ref{prop.acute among all} we will show that, with the possible exception of acute triangles that have an angle of measure $\pi/j$ for $j=3,5,7$, every non-obtuse triangle is in fact uniquely determined by its characteristic polynomial within the set of \emph{all} triangles.  For this proof, we require the following two lemmas.

\begin{lemma}\label{lem.sine property}
If $T$ has exactly one odd angle, say $\alpha_1$, then the characteristic polynomial of $T$ determines the quantity
$$\frac{\sin(\alpha_2)+\sin(\alpha_3)}{\sin(\alpha_1)}. $$
    \end{lemma}

\begin{proof} 
By the Law of Sines together with the convention for the order of enumerating angles and side lengths, 
\[ \frac{\sin(\alpha_1)}{\ell_3} = \frac{\sin(\alpha_2)}{\ell_1} = \frac{ \sin(\alpha_3)}{\ell_2} \implies \frac{\sin(\alpha_2)+\sin(\alpha_3)}{\sin(\alpha_1)}  =  \frac{\ell_1+\ell_2}{\ell_3}.\]
By Remark \ref{rem:oneodd_tri}, $P_T(t)$ determines $\ell_3$ and $\ell_1+\ell_2$.
\end{proof}

\begin{lemma}\label{lem.angle dec}
Let $T$ and $T'$ be  triangles each having exactly one odd angle $\alpha_1$, respectively $\alpha_1'$, and with the same characteristic polynomial. Assume that $$\max\{\alpha_2',\alpha_3'\}>\max\{\alpha_2,\alpha_3\}.$$
Then 
\beq \min\{\alpha_2',\alpha_3'\}<\min\{\alpha_2,\alpha_3\} \textrm{ and } \alpha_1'<\alpha_1. \label{eq:lem.angle dec} \eeq 
\end{lemma}

\begin{proof}
As observed in the proof of Theorem \ref{th:triangles123}, there exists an ellipse $E$ such that the vertices $\alpha_1$ and $\alpha_1'$ of $T$ and $T'$ lie on $E$, and the other two vertices of each of these triangles are at the foci of $E$.   We may assume that $E$ is centered at the origin with foci on the $x$-axis, so $E$ is given by $\frac{x^2}{a^2}+\frac{y^2}{b^2}=1$ with $a,b>0$.   Let $(\pm c,0)$ denote the foci, where $c>0$.

Consider all triangles that have one vertex on $E$ and the other vertices at the foci.  Since we only care about the congruence classes of the triangles, we can assume that the vertex $v$ on $E$ lies in the first quadrant or at $(0,b)$.  Let $v$ move continuously along $E$ starting at $(0,b)$ and heading towards  $(a,0)$. As observed in the proof of Theorem \ref{th:triangles123}, the knowledge of a single angle uniquely determines the triangle up to congruence.  Consequently, as $v$ moves within a quadrant, each of the angles is a monotone function.  Denote by $T(v)$ the triangle formed by $v$ and the foci as depicted in Figure \ref{fig:T_of_v}.    Without loss of generality, let $\alpha_2$ be the angle at $(-c,0)$ and $\alpha_3$ be the angle at $(c,0)$.  Then, when $v$ is at $(0,b)$, the triangle is isosceles, and $\alpha_2=\alpha_3$.  

As $v$ moves towards $(a,0)$, both $\alpha_1$ and $\alpha_2$ decrease monotonically while $\alpha_3$ increases monotonically. 
Consequently, $\max \{ \alpha_2, \alpha_3\} = \alpha_3$, and similarly $\max \{ \alpha_2', \alpha_3'\} = \alpha_3'$.  If $\alpha_3' > \alpha_3$, this means that the vertex $v'$ is closer to $(a,0)$ than the vertex $v$ as shown in Figure \ref{fig:T_of_v}.  It therefore follows from the monotonicity of the angles that $\alpha_2'<\alpha_2$ and $\alpha_1' < \alpha_1$.
\end{proof}
\begin{figure}[h]
\centering 
\resizebox{6cm}{!}{\begin{tikzpicture}[dot/.style={draw,fill,circle,inner sep=1pt}]
  \def\a{5} 
  \def\b{3} 
  \def\angle{45} 
   \def\bangle{25}
  \draw (0,0) ellipse ({\a} and {\b});
  \draw (-\a,0) 
    -- (\a,0) ;
  \draw (0,-\b)
    -- (0,\b) ;
  \node[dot] (F1) at ({-sqrt(\a*\a-\b*\b)},0) {};
  \node[dot] (G1) at ({-sqrt(\a*\a-\b*\b)},0) {};
\node[above] at ({-sqrt(\a*\a-\b*\b)},0)  {\small $(-c,0)$};  

  \node[dot] (F2) at ({+sqrt(\a*\a-\b*\b)},0) {};
  \node[dot] (G2) at ({+sqrt(\a*\a-\b*\b)},0) {};
  \node[below] at ({sqrt(\a*\a-\b*\b)},0)  {\small $(c,0)$}; 
  \node[dot,label={\angle:$v$}] (X) at (\angle:{\a} and {\b}) {};
   \node[dot,label={\angle:$v'$}] (Y) at (\bangle:{\a} and {\b}) {};
  \draw (F1) -- (X) (X) -- (F2);
    \draw (G1) -- (Y) (Y) -- (G2);
    \coordinate[label={below left:$T(v)$}] (L2) at (0,1.6);
    \coordinate[label={below right:$T(v')$}] (L1) at (2.7,1.6);
  
\end{tikzpicture}}
\caption{The angle of $T(v)$ at vertex $v$ is $\alpha_1$ and it determines $T(v)$ uniquely.} \label{fig:T_of_v}
\end{figure}
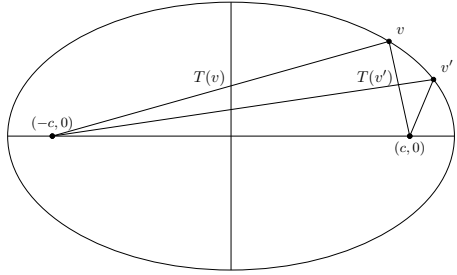

\begin{prop} \label{prop.acute among all}
Suppose $T$ satisfies either of the following conditions: 
\begin{itemize} 
\item $T$ is an acute triangle with no angle in $\{\frac{\pi}{3}, \frac{\pi}{5}, \frac{\pi}{7}\}$. 
\item $T$ is a right triangle.
\end{itemize} Then any triangle that has the same characteristic polynomial as $T$ is congruent to $T$.  
\end{prop}

 \begin{proof} 
Under either of these assumptions, $T$ is not obtuse.  So, any non-obtuse triangle that has the same characteristic polynomial as $T$ is congruent to $T$ by Theorem \ref{th:non_obtuse}.  Thus, we only need to analyze the possibility that an obtuse triangle has the same characteristic polynomial as $T$.  Moreover, if $T$ does not have any odd angles, then by Theorem \ref{th:triangles123}, any triangle that has the same characteristic polynomial as $T$ is congruent to $T$.  In both cases of the proposition, $T$ cannot have two odd angles. To see this, in the first case, the third angle would necessarily be at least $\frac{7\pi}{9}$, a contradiction, while in the case of a right triangle the sum of $\frac \pi 2$ and two odd angles cannot equal $\pi$. So, as in the proof of Theorem~\ref{th:non_obtuse}, we may suppose $T$ has precisely one odd angle. We therefore assume that $\alpha_1$ is odd.  By Theorem \ref{th:triangles123}, the characteristic polynomial determines the number of odd angles, so we assume that $T'$ is obtuse, has the same characteristic polynomial as $T$, and has one odd angle denoted $\alpha_1'$.   

To prove the proposition, it thus suffices to consider two cases:
\begin{itemize}
    \item[(a)] $T$ is a non-obtuse triangle (either an acute triangle or a right triangle) with odd angle $\alpha_1\notin\{\frac{\pi}{3}, \frac{\pi}{5}, \frac{\pi}{7}\}$.
    \item[(b)] $T$ is a right triangle with $\alpha_1\in\{\frac{\pi}{3}, \frac{\pi}{5}, \frac{\pi}{7}\}$.
\end{itemize}

Consider case (a).  We must have $\alpha_1\leq \frac{\pi}{9}$.   Since $T$ has no obtuse angles, and since $\frac{\pi}{2}-\frac{\pi}{9}=\frac{7\pi}{18}$, we must have 
$$\frac{7\pi}{18}\leq \alpha_2,\alpha_3\leq \frac{\pi}{2}.$$
Thus $\frac{\pi^2}{2\alpha_i} \in [\pi, \frac{9\pi}{7}]$ for $i=2,3.$ 
Let $\alpha_3'$ be the obtuse angle of $T'$. Then we immediately obtain $\frac{\pi^2}{2\alpha_3'} \in (\frac{\pi}{2},\pi).$  However, by Remark \ref{rem:oneodd_tri}, we know that $|c(\alpha_3')|$ must equal one of $|c(\alpha_2)|$ or $|c(\alpha_3)|$, which implies that $\frac{\pi^2}{2\alpha_3'} \in [\frac{5\pi}{7},\pi).$ 
We therefore have 
$\alpha_3' \in (\frac{\pi}{2},\frac{7\pi}{10}]$, so that \[ \alpha_1'+\alpha_2'\geq \frac{3\pi}{10}.\] 
Since $\alpha_2, \alpha_3 \leq \frac{\pi}{2} < \alpha_3'$, 
Lemma~\ref{lem.angle dec} implies 
$\alpha_1' < \alpha_1 \leq \frac \pi 9$, and hence $\alpha_1' \leq \frac{\pi}{11}$.
Thus 
\beq \alpha_2'\geq \frac{3\pi}{10}-\frac{\pi}{11} >\frac{\pi}{5}\ \  \text{so that} \ \  \pi < \frac{\pi^2}{2\alpha_2'} < \frac{5\pi}{2}, \label{eq:beta2'_range} \eeq 
where we used that $\alpha_2' < \frac{\pi}{2}$ since $\alpha_3'$ is obtuse.  
Observe that $c(\alpha_2), c(\alpha_3)$ and $c(\alpha_3')$ are all negative.  Since the sign of $c(\alpha_2)c(\alpha_3)$ equals that of $c(\alpha_2')c(\alpha_3')$, we have $c(\alpha_2')<0$. Combining this sign information with 
Remark \ref{rem:oneodd_tri} yields that $ c(\alpha_2')=c(\alpha_j)<0$ for some choice of $j\in\{2,3\}$.  For either choice of $j$, we have $\frac{\pi^2}{2\alpha_j} \in [\pi, \frac{9\pi}{7}]$ and thus $\frac{\pi^2}{2\alpha_2'} \in (\pi, \frac{9\pi}{7}].$ Hence $\alpha_2' \in [\frac{7\pi}{18},\frac{\pi}{2}).$  But $|c|$ is bijective on $(\frac{\pi}{3},\frac{\pi}{2})$ by Lemma \ref{lem: |c|}, so $\alpha_2'$ equals one of $\alpha_2$ or $\alpha_3$.  By Theorem \ref{th:triangles123}(c), $T=T'$.  

We now show that the proposition holds in case (b).  Continuing to let $\alpha_1$ and $\alpha_1'$ denote the odd angle in $T$ and $T'$, respectively,  there exist $k,k'\in \Z_{+}$ such that 
$$\alpha_1= \frac{\pi}{2k+1} \mbox{ and } \alpha'_1= \frac{\pi}{2k'+1}, $$ 
where $k\in \{1,2,3\}$.  Let $\alpha_2= \frac{\pi}{2}$ and $\alpha_3=\frac{2k-1}{4k+2}\pi$.
Since the characteristic polynomial will have no constant term, we know that $T'$ has at least one even angle.  If this angle is $\frac{\pi}{2}$, then $T=T'$, so we set $\alpha_2'= \frac{\pi}{2m}$ for some positive integer $m\geq 2$. The characteristic polynomials are
\[ P_T (t) =\cos(t) -c(\alpha_3)
\cos (|1 - 2\ell_3|t)= \cos(t) +(-1)^{m}c(\alpha'_3)
\cos (|1 - 2\ell'_3|t)= P_{T'} (t).\] 
Thus $\ell_3=\ell_3'$, and, since $\alpha_3'$ is obtuse and $\frac{\pi^2}{2\alpha_3'} \in (\frac{\pi}{2}, \pi)$, we have 
$$c(\alpha'_3)=-|c(\alpha'_3)|=-|c(\alpha_3)|= -\left|\cos \left(\frac{2k+1}{2k-1}\pi\right)\right|.$$ 
Since $|c|$ is injective on the set of obtuse angles by Lemma \ref{lem: |c|}, we may compute $\alpha_3'$ for each case $k=1,2,3$. 

\noindent \textbf{Case $\textbf{k=1}$.} In this case we have that 
\[c(\alpha_3') = -|\cos(3\pi)|=-1,\]
so that $\frac{\pi^2}{2 \alpha_3'} = \pi$ and $\alpha_3' = \frac{\pi}{2}=\alpha_2$.  By Theorem \ref{th:triangles123}(c), $T=T'$. \\
\textbf{ Case $\textbf{k=2}$.}  In this case we have that 
\[c(\alpha_3') = -\left|\cos\left(\frac{5\pi}{3}\right)\right|=-\frac{1}{2},\]
so that $\frac{\pi^2}{2 \alpha_3'} = \frac{2\pi}{3}$ and $\alpha_3' = \frac{3\pi}{4}$.   Thus, \[ \alpha'_1 = \pi - \frac{3\pi}{4} - \frac{\pi}{2m} = \frac{(m-2)\pi}{4m}=\frac{\pi}{2k'+1},\]  so $(m-2)(2k'+1) = 4m$, and $4$ divides $m-2$.  
Thus $m\geq 6$, and consequently $\alpha_2' = \frac{\pi}{2m} \leq \frac{\pi}{12}$. Hence 
\[ \alpha_1' = \pi - \alpha_2' - \alpha_3' \geq \frac{\pi}{6}. \]
So $\alpha_1'$ is an odd angle between $\frac{\pi}{6}$ and $\frac{\pi}{4}$, which implies $\alpha'_1 = \frac{\pi}{5} =\alpha_1.$
Then $T=T'$.\\
\noindent \textbf{ Case $\textbf{k=3}$.}  
In this case we have that 
\[c(\alpha_3') = -\left|\cos\left(\frac{7\pi}{5}\right)\right|=\cos\left(\frac{3\pi}{5}\right),\]
so that $\frac{\pi^2}{2 \alpha_3'} = \frac{3\pi}{5}$ and $\alpha_3' = \frac{5\pi}{6}$. 
This implies that 
$$\pi= \frac{\pi}{2k'+1} + \frac{\pi}{2m}+\frac{5\pi}{6}.$$

Thus 
\[\frac{1}{2k'+1} + \frac{1}{2m} = \frac{1}{6}.\]
At least one of the two summands must be greater than or equal to $\frac{1}{12}$ while each is less than $\frac{1}{6}$.     Thus either $m\in\{4,5,6\}$ or $k'\in\{3,4,5\}$, and one easily checks that the only possibilities for $T'$ are
$$T'_1:=\left(\frac{\pi}{7},\frac{\pi}{42}, \frac{5\pi}{6}\right),~T'_2:=\left(\frac{\pi}{9},\frac{\pi}{18}, \frac{5\pi}{6}\right)~\mbox{ and } T'_3=\left(\frac{\pi}{15},\frac{\pi}{10}, \frac{5\pi}{6}\right). $$

To complete the proof, we use Lemma \ref{lem.sine property}. We have $\alpha_1=\frac{\pi}{7}$, $\alpha_2=\frac{\pi}{2}$, and $\alpha_3=\frac{5\pi}{14}$, so 
$$L_T:=\frac{\sin\alpha_2+\sin\alpha_3}{\sin \alpha_1}=\frac{\sin \frac{\pi}{2}+\sin\frac{5\pi}{14}}{\sin \frac{\pi}{7}}\approx 4.38.$$
On the other hand,
$$L_{T'} \in \left\{ \frac{\sin \frac{\pi}{42}+\sin\frac{5\pi}{6}}{\sin \frac{\pi}{7}}\approx 1.32, ~~~\frac{\sin \frac{\pi}{18}+\sin\frac{5\pi}{6}}{\sin \frac{\pi}{9}}\approx 1.97,~~~\frac{\sin \frac{\pi}{10}+\sin\frac{5\pi}{6}}{\sin \frac{\pi}{15}}\approx 3.89\right\}.$$
Since we have assumed $T$ and $T'$ each have precisely one odd angle and have the same characteristic polynomial, this is a contradiction to Lemma \ref{lem.sine property}.
 \end{proof}

\section{Special classes of quadrilaterals}\label{s:quad}
We consider three special classes of convex quadrilaterals:  rectangles, parallelograms, and kites.   Our strongest results are for rectangles:  we find that every rectangle is uniquely distinguished by its characteristic polynomial within the set of all convex quadrilaterals. We then address whether parallelograms are mutually distinguishable and similarly for kites.

\subsection{Rectangles}\label{ss:rectangles}

The characteristic polynomial of a rectangle $R$ of perimeter equal to one and with edge lengths $\ell$ and $\ell'$, with $\ell\leq\ell'$, is given by
\beq\label{eq:char rect}P_R(t)= \cos(t)+2\cos(2\ell t)+2\cos(2\ell' t)+\cos(2(\ell'-\ell)t)+2.\eeq
In the special case of a square, so $\ell=\ell'$, the characteristic polynomial simplifies to 
\beq\label{eq:char square} P_R(t)=\cos(t) + 4\cos(t/2) +3.\eeq  

\begin{thm}\label{thm:rectangle}
Let $R$ be a rectangle. Then $R$ is uniquely determined within the class of all convex quadrilaterals and triangles by its characteristic polynomial.
\end{thm}

The fact that any two rectangles are mutually distinguishable by their Steklov spectrum was proven by other methods in \cite[Corollary 1.8]{cuboid}.

\begin{proof} 
Let $R$ be a rectangle of perimeter one, and let $\Om$ be a triangle or quadrilateral with $P_\Om=P_R$.     The sum of the coefficients of all the cosine terms together with the constant term in the characteristic polynomial is equal to 8.    This is strictly larger than that of any triangle.  Thus $\Om$ is a quadrilateral.   

Denote the angles of $\Om$ by $\al_1, \dots, \al_4$ and the vector of edge lengths of $\Om$ by $\bell_*$.  To show that $\Om$ is a rectangle, it suffices to show that $|c(\al_j)|=1$ for all $j$, i.e., that all angles are even.  Indeed, rectangles are the only quadrilaterals with all angles even since the angle sum of a quadrilateral is $2\pi$.  Recall from Definition~\ref{def: char poly} that the coefficient $a_{[\bxi]}$ is equal to 1  when $[\bxi]=[(1,1,1,1)]$  
and $a_{[\bxi]}$ is a product of some of the $c(\al_j)$'s otherwise.      

Let $A:=\{[\bxi]\in\{\pm 1\}^4/\sim:\,\bxi\cdot\bell_*=0\}$ and $B=\{[\bxi]\in\{\pm 1\}^4/\sim:\,\bxi\cdot\bell_*\neq0\}$.   
For simplicity, first assume that the comparison rectangle $R$ is not a square. Then the constant term in $P_\Om$ satisfies
\[2=\sum_{[\bxi]\in A}\, a_{[\bxi]}\cos(|\bxi\cdot\bell_* |t) -\prod_{j=1}^4\,s(\alpha_j).\] 
Since at most two angles of $\Om$ can be odd,  we have $\left|\prod_{j=1}^4\,s(\al_j) \right|<1$.  The fact that $|a_{[\bxi]}|\leq 1$ for all $\bxi$ implies that $|A|\geq 2$ and thus $|B|\leq 6$.  Since the sum of the coefficients of the cosine terms in $P_{\Om}$ is six, we must have $a_{[\bxi]}=1$ for each $[\bxi]\in B$.  One can check that for each choice of $j$, $c(\al_j)$ occurs as a factor in exactly half of the $a_{[\bxi]}$'s, implying that $|c(\alpha_j)|=1$ for all $j$.  As noted above, it follows that $\Om$ is a rectangle.    A similar argument goes through when the comparison rectangle $R$ is a square.

Now that we know $\Om$ is a rectangle, it is straightforward to read off the edge lengths from the cosine frequencies occurring in $P_{\Om}-\cos(t)$: the largest frequency yields $\ell'$ and then $\ell =\frac{1}{2}-\ell'$.    
\end{proof}

\subsection{Parallelograms}

Let $\Om$ be a parallelogram of perimeter one with angles $\alpha_1(=\alpha_3)$ and $\alpha_2(=\alpha_4)$ and edge lengths $\ell$ and $\ell'$, where $\ell\leq \ell'$.  The characteristic polynomial is given by

\beq
P_{\Om}(t) = \cos(t) &+& 2c(\al_1)c(\al_2)\cos(2\ell t) + 2c(\al_1)c(\al_2)\cos(2\ell' t) \nn \\
&+& c(\al_1)^2 c(\al_2)^2 \cos\left(2|\ell'-\ell |t\right) + C
\label{eq:char parallelogram}
\eeq
where $C$ is given by 

\beq\label{char par const}
C= c(\alpha_1)^2 +c(\al_2)^2 \, -\,s(\al_1)^2s(\al_2)^2\,= \,2c(\alpha_1)^2 +2c(\al_2)^2 - c(\al_1)^2 c(\al_2)^2 -1.
\eeq

In case $\Om$ is a rhombus, so $\ell=\ell' =\frac{1}{4}$, Equation~\eqref{eq:char parallelogram} simplifies to 

\beq\label{eq:char rhombus} P_{\Om}(t) = \cos(t) + 4c(\al_1)c(\al_2)\cos\left(\frac{t}{2} \right)  
+ 2c(\alpha_1)^2 +2c(\al_2)^2 \, -1.
\eeq

If $\Om$ is any parallelogram (possibly a rhombus) with a pair of odd angles, say $\alpha_2$, then Equation~\eqref{eq:char parallelogram} reduces to 

\beq\label{eq:char par odd}
P_{\Om}(t) = \cos(t) \,+\, c(\alpha_1)^2  \, -\,s(\al_1)^2\,=\, \cos(t) \,+\, \cos\left(\frac{\pi^2}{\alpha_1}\right).
\label{eq:char paral odd}
\eeq

\begin{thm}\label{thm:parallelograms}
Let $\Om$ be a parallelogram.
\begin{itemize}
\item[(a)] 
If $\Om$ has no odd angles, then there is at most one other parallelogram with the same characteristic polynomial as $\Om$.   Moreover, the characteristic polynomial determines the edge lengths.
\item[(b)] If $\Om$ has a pair of odd angles, then the characteristic polynomial uniquely determines the angles of $\Om$ but gives no information about the edge lengths.
\end{itemize}
\end{thm}

\begin{proof}
(a) Since $\Om$ has no odd angles, Equations~\eqref{eq:char parallelogram} and \eqref{eq:char rhombus} enable us to read off the longer edge length $\ell'$ (or the unique edge length $\frac{1}{4}$ in the case of a rhombus) from the maximal cosine frequency occurring in $P_{\Om}(t) -\cos(t)$, and then $\ell=\frac{1}{2}-\ell'$.    

It remains to show that the angles are determined up to two possibilities.   Write $x=c(\alpha_1)$ and $y=c(\al_2)$, and observe that $A:=xy$ and $B:=x^2+y^2$ are determined by the characteristic polynomial.  The intersection of the hyperbola $xy=A$ and the circle $x^2+y^2=B$ is of the form $\{(a,b), (-a,-b), (b,a), (-b,-a)\}$ for some $a,b$.    Thus we can read off $(c(\al_1), c(\al_2))$ up to order and global change of sign. By Theorem~\ref{thm:rectangle}, we may assume that $\Om$ is not a rectangle and thus that one angle of $\Om$ is obtuse.   Obtuse angles $\al$ are uniquely determined by $|c(\al)|$, so the obtuse angle is determined up to two possibilities and thus $\Om$ is too.  

(b)   Without loss of generality, suppose $\al_2$ is odd so that $P_\Om(t)$ is given by \eqref{eq:char paral odd}.  Observe that the map $\alpha_1\to \cos\left(\frac{\pi^2}{\alpha_1}\right)$ is injective on the interval $\left(\frac{\pi}{2},\pi\right)$.      Since $\alpha_1$ is necessarily obtuse, we can thus read off $\alpha_1$ from $P_\Om(t)$ and then $\al_2=\pi-\al_1$.    
\end{proof}

\subsection{Kites}\label{ss:kites}
  We now consider convex quadrilaterals known as kites; see Figure \ref{fig:kite}.  We will denote the edge lengths by $\ell$ and $\ell'$ and denote the angles by $\alpha$, $\gamma$, and $\gamma'$.   Rhombi are precisely the equilateral kites, i.e., kites satisfying $\ell=\ell'$ and then necessarily $\gamma=\gamma'$.  We choose our labeling so that 
\beq\label{eq:kite labels}\ell\leq \ell'\mbox{\,\,and then\,\,} \gamma\geq\gamma'.
\eeq

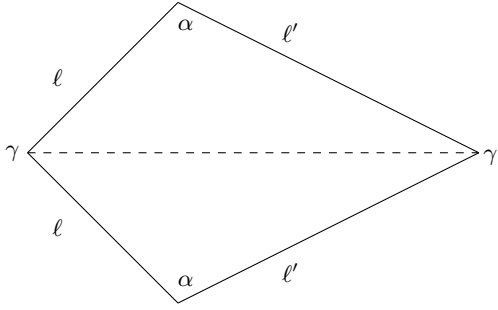
\begin{figure}[h] \centering 
\begin{tikzpicture}
\draw   (-2,0) --  (0,-2);
\draw  (-2, 0) -- (0, 2); 
\draw (0, -2) -- (4, 0); 
\draw (0, 2) -- (4, 0); 
\draw [dashed] (-2,0) -- (4,0); 
\node at (-2.2,0) {\small $\gamma$}; 
\node at (-1.6, 1) {\small $\ell$}; 
\node at (-1.6, -1) {\small $\ell$}; 
\node at (1.5, -1.6) {\small $\ell'$}; 
\node at (1.5, 1.6) {\small $\ell'$};
\node at (4.2, 0) {\small $\gamma'$}; 
\node at (0.1, -1.7) {\small $\alpha$}; 
\node at (0.1, 1.7) {\small $\alpha$}; 

\end{tikzpicture}
\caption{A convex kite is a quadrilateral that has two pairs of adjacent sides that are of equal lengths. Consequently, splitting the kite along the dashed segment, one obtains two congruent triangles.} \label{fig:kite} 
\end{figure}

We assume all kites under consideration have perimeter one.  Following Definition ~\ref{def: char poly}, the characteristic polynomial of a kite $K$ simplifies to 

\beq
P_K (t) = \cos(t) &+& 2c(\alpha)c(\gamma)\cos(2\ell' t) \,+ \,2c(\alpha)c(\gamma')\cos(2\ell t) \,+\, c(\alpha)^2\cos\left(2|\ell'-\ell|t\right) \nn \\
&+& \left[c(\gamma)c(\gamma')\, +\, c(\alpha)^2c(\gamma)c(\gamma')\,-\, s(\alpha)^2s(\gamma)s(\gamma')\right].\label{eq:char_kite}
 \eeq
 
We will see that for almost all kites $K$, the characteristic polynomial of $K$ determines $K$ up to at most three possibilities within the set of all kites, and a substantial collection of kites are uniquely determined.  There are precisely two cases in which two of the values $\ell, \ell'$, and $\ell'-\ell$ coincide:  when $K$ is a rhombus and therefore $\ell=\ell'=\frac{1}{4}$, or when $\ell=\ell'-\ell$, in which case $(\ell, \ell')=(\frac{1}{6}, \frac{1}{3})$.  Also note that a kite can have at most two odd angles and, in that case, the two odd angles must be opposite each other.  We will be interested in the minimal amount of geometric data needed to determine a kite; the elementary proof of the following lemma is left to the reader.

\begin{lemma}\label{lem:angle determines}
If two kites have the same edge lengths and at least one angle in common, then they are congruent.
\end{lemma}
 
We first consider kites with no odd angles and no length coincidences; almost all kites are in this class.

\begin{prop}\label{kites_no_odd_angles}
Let $K$ be a kite with no odd angles and suppose that $\ell' \neq \ell, 2\ell$.  Then there is at most one other kite with the same characteristic polynomial.  
Moreover, if $K$ has only one acute angle, then $K$ is uniquely determined by its characteristic polynomial. 
\end{prop}

\begin{proof}
Since $K$ has no odd angles and we have excluded coincidences among $\ell$, $\ell'$, and $\ell'-\ell$, there are three distinct cosine terms in $P_K(t)-\cos(t)$ as given by \eqref{eq:char_kite}.  The largest cosine frequency is $2\ell'$. Thus the characteristic polynomial determines $\ell'$ and hence also $\ell=\frac 1 2 -\ell'$.  

The coefficients of the cosine terms enable us to read off first $|c(\alpha)|$ and then $|c(\gamma)|$ and $|c(\gamma')|$.   Every quadrilateral that is not a rectangle (in particular, every kite that is not a square) has at least one angle greater than $\frac{\pi}{2}$.    Thus since $\gamma' \leq \gamma$, at least one of $\gamma$  or $\alpha$ exceeds $\frac{\pi}{2}$.  From the fact that the map $\beta\mapsto |c(\beta)|$ is injective on $(\frac{\pi}{2},\pi)$, we can therefore determine one of $\alpha$ or $\gamma$ from the characteristic polynomial.   The first statement now follows from Lemma~\ref{lem:angle determines}.   If $\alpha$ and $\gamma$ both exceed $\frac{\pi}{2}$ then the two resulting possibilities must coincide, so the kite is uniquely determined.
\end{proof}

Next we consider kites with one odd angle.  Since angles in a rhombus come in pairs, such a kite cannot be a rhombus.

\begin{prop}\label{kites_one_odd_angle}
Let $K$ be a kite of perimeter one with exactly one odd angle and suppose that $(\ell, \ell')\neq(\frac{1}{6}, \frac{1}{3})$. Then at most two other kites have the same characteristic polynomial as $K$. Moreover:
\begin{enumerate}
\item[(a)] If the odd angle of $K$ is the larger of the two opposite non-congruent angles and $(\ell,\ell')\neq \left(\frac{1}{10}, \frac{2}{5}\right)$, then $K$ is uniquely determined within the set of all kites by its characteristic polynomial.
\item[(b)] If the odd angle is the smaller of the two opposite non-congruent angles and $(\ell,\ell')\neq \left(\frac 1 5, \frac{3}{10}\right)$, then at most one other kite has the same characteristic polynomial as $K$.   
\end{enumerate}

\end{prop}

\begin{proof}
The odd angle in $K$ is either $\gamma$ or $\gamma'$.  In both cases, $P_K(t) - \cos t$ has exactly two distinct cosine terms:
\beq\label{eq:g} P_K(t)-\cos(t)=2c(\alpha)c(\gamma')\cos(2\ell t)\,+\, c(\alpha)^2\cos\left(2|\ell'-\ell|t\right) \,+\, \rm
{constant}
\eeq
if the odd angle is $\gamma$, and
\beq\label{eq:g'} P_K(t)-\cos(t)=2c(\alpha)c(\gamma)\cos(2\ell' t)\,+\, c(\alpha)^2\cos\left(2|\ell'-\ell|t\right) \,+\, \rm
{constant}.\eeq
if the odd angle is $\gamma'$.

The requirement that $P_{K}(t)-\cos(t)$ has exactly two distinct cosine frequencies, together with our exclusion of the case $(\ell, \ell')=(\frac{1}{6}, \frac{1}{3})$, means that any kite with the same characteristic polynomial as $K$ must also have exactly one odd angle.  We consider the two possibilities in turn.  First, suppose the odd angle is the larger of the two non-congruent angles.  Then summing the cosine frequencies on the right side of Equation \eqref{eq:g}, we obtain $2\ell'$. We determine $\ell'$ and then $\ell = \frac{1}{2}-\ell'$. Since the odd angle is the larger of the two non-congruent angles, the repeated angle $\alpha$ must be obtuse and hence is determined by $|c(\alpha)|$. Lemma~\ref{lem:angle determines} then shows that within the class of kites with one odd angle that is the larger of the two non-congruent angles, each kite is uniquely determined by its characteristic polynomial.  

Next, suppose the odd angle is the smaller of the two non-congruent angles.  Then summing the cosine frequencies on the right side of Equation \eqref{eq:g'}, we obtain $4\ell'-2\ell$.  We add $2\ell'+2\ell=1$ and again determine $\ell'$ and hence $\ell$.  We have $\alpha >\frac{\pi}{3}$ and the coefficient of the final cosine term yields $\alpha$ up to two possibilities.    We can then apply Lemma~\ref{lem:angle determines} to conclude that within the class of kites with one odd angle that is the smaller of the two non-congruent angles, each kite is determined up to two possibilities by its characteristic polynomial.  A priori, a kite with the same characteristic polynomial as $K$ may not have the same type of odd angle as $K$, meaning that there could be at most three kites with exactly one odd angle and the same characteristic polynomial.

To prove (a) and (b), it remains only to show that, with the exception of the kites excluded by the given hypotheses, the characteristic polynomial of a kite with one odd angle determines whether the odd angle is the larger or smaller of the two opposite non-congruent angles.   Let $K_1$, respectively $K_2$, be kites with odd angle that is larger, respectively smaller, than its opposite angle, and denote the edge lengths of $K_j$ by $(\ell_j, \ell'_j)$  for $j=1,2$.  Then $P_{K_1}(t) - \cos t$ is given by Equation~\eqref{eq:g} and $P_{K_2}(t) - \cos t$ is given by Equation~\eqref{eq:g'}.  In order for these to be equal and taking into account that $2\ell_1 < \frac 1 2 < 2\ell_2'$, we must have $\ell_1=\ell_2'-\ell_2$ and $\ell_2'=\ell_1'-\ell_1$.  Since both kites have perimeter one, it follows that $(\ell_1, \ell_1')=\left(\frac{1}{10}, \frac{2}{5}\right)$ and $( \ell_2, \ell_2') =\left(\frac 1 5, \frac{3}{10}\right)$, completing the proof.
\end{proof}

The case of kites with two odd angles is more subtle.  If a kite $K$ has two equal odd angles, then these are necessarily the angles $\alpha$ and the only cosine term in $P_K(t)$ is $\cos (t)$.  Thus $P_K(t)$  provides very little information about the geometry of $K$, and we will not study such kites further.   If $K$ is a kite with two unequal odd angles $\gamma$ and $\gamma'$, then $P_K(t) - \cos (t)$ retains one cosine term and hence more information; in particular, 
\beq
P_K(t)-\cos(t)=c(\alpha)^2\cos\left(2|\ell'-\ell|t\right) \,+\, \text{constant}.
\label{eq:char_2unequalodd}
\eeq
However, there are certain length pairs that cause the cosine frequency $2|\ell'-\ell|$ to match those of certain kites with no or exactly one odd angle:
\begin{enumerate}[(a)]
\item If $K$ satisfies $(\ell, \ell') = (\frac{1}{6}, \frac{1}{3})$, then it could have the same cosine frequencies as a kite $K'$ in which $\gamma$ is the unique odd angle, $(\ell, \ell')=(\frac{1}{6}, \frac{1}{3})$ and $c(\gamma')\neq -\frac{1}{2}c(\alpha)$:
$$P_{K'}(t)-\cos(t)=[2c(\alpha)c(\gamma') + c(\alpha)^2]\cos(t/3) \,+\, \text{constant}.$$
\item If $K$ satisfies $(\ell, \ell') = (\frac{1}{12}, \frac{5}{12})$, then it could have the same cosine frequencies as a kite $K'$ in which there are no odd angles, $(\ell, \ell')=(\frac{1}{6}, \frac{1}{3})$ and $c(\gamma')= -\frac{1}{2}c(\alpha)$:
$$P_{K'}(t)-\cos(t)=2c(\alpha)c(\gamma)\cos(2t/3) \,+\, \text{constant}.$$
\item If $K$ satisfies $(\ell, \ell') = (\frac{1}{8}, \frac{3}{8})$, then it could have the same cosine frequencies as a rhombus $K'$ with no odd angles:
$$P_{K'}(t)-\cos(t)=2c(\alpha)[c(\gamma)+c(\gamma')]\cos(t/2)\,+\, \text{constant}.$$    
\end{enumerate}
We will exclude these possibilities from our consideration of kites with unequal odd angles.

\begin{prop}\label{prop:kites_two_unequal_odd} Let $K$ be a kite of perimeter one with two unequal odd angles.   Suppose its lengths satisfy 
$(\ell,\ell')\notin\{(\frac{1}{6}, \frac{1}{3}), (\frac{1}{12}, \frac{5}{12}),(\frac{1}{8}, \frac{3}{8})\}$.  Then $K$ is uniquely determined amongst all kites by its characteristic polynomial.  
\end{prop}

\begin{proof}
The characteristic polynomial of $K$ as given in \eqref{eq:char_2unequalodd} determines $\ell'-\ell$ and hence the pair $(\ell,\ell')$.   We can also read off $|c(\alpha)|$. Since $\gamma$ and $\gamma'$ are both odd, $\alpha$ must be obtuse and hence $\alpha$ is uniquely determined from $|c(\alpha)|$.   We can thus apply Lemma~\ref{lem:angle determines} to complete the proof.
\end{proof}

We conclude our discussion of kites by considering some of the equilateral ones.

\begin{thm} \label{th:rhombus}
 Let $K$ be a rhombus with no odd angles.  Then there are at most two other kites with the same characteristic polynomial.
\end{thm}

\begin{proof} 
In Theorem~\ref{thm:rectangle}, we proved that squares are uniquely determined by their characteristic polynomial among all convex quadrilaterals and triangles.  So, assume that $K$ is not a square.  From the preceding discussion, the only kites that might have the same characteristic polynomial as $K$ are other rhombi with no odd angles, and kites with unequal odd angles $\gamma$ and $\gamma'$ and $(\ell, \ell')=(\frac{1}{8}, \frac{3}{8})$.  We have already shown in Theorem~\ref{thm:parallelograms}(a) that there is at most one other rhombus with no odd angles with the same characteristic polynomial as $K$.

The constant term for a kite with unequal odd angles $\gamma$ and $\gamma'$ and $(\ell, \ell')=(\frac{1}{8}, \frac{3}{8})$ is $-s(\al)^2s(\gamma)s(\gamma')=\pm s(\al)^2$.  The coefficient of the $\cos (t/2)$ term is $c(\al)^2$.  Thus the quantities $c(\al)^2 \pm s(\al)^2$ are determined by the characteristic polynomial: for one sign choice the value is one, and for the other sign choice the value is $\cos(\frac{\pi^2}{\alpha})$.  Since $\alpha$ is necessarily obtuse, we have $\frac{\pi^2}{\al} \in (\pi, 2\pi)$, and $\cos \theta$ is injective on $(\pi, 2\pi)$.  Hence we recover $\alpha$ and there is one possible non-equilateral kite with the same characteristic polynomial as $K$.   
\end{proof}

\section{Regular polygons} \label{s:regular} 
We now focus on regular polygons. The following notation will be useful.

\begin{nota}\label{nota:A}
We define the following subsets of $(0,\pi)$:
\[\A^-:=\{\alpha\in (0,\pi):  c(\al)<0\}=\bigsqcup_{j=1}^\infty\, \left(\frac{\pi}{4j-1}, \frac{\pi}{4j-3}\right)=\left(\frac{\pi}{3},\pi\right)\cup \left(\frac{\pi}{7},\frac{\pi}{5}\right)\cup\dots\]
\[ \A^+:=\{\alpha\in (0,\pi):  c(\al)>0\}=\bigsqcup_{j=1}^\infty\, \left(\frac{\pi}{4j+1}, \frac{\pi}{4j-1}\right)= \left(\frac{\pi}{5},\frac{\pi}{3}\right)\cup \left(\frac{\pi}{9},\frac{\pi}{7}\right)\cup\dots\]
and 
\[\A^0:=\{\alpha\in (0,\pi):  c(\al)=0\} = \{\rm{odd\,\,angles}\}.\]
\end{nota}
We have \[(0,\pi) =\A^-\sqcup\A^+\sqcup\A^0.\]
We emphasize that $\A^-$ is the largest of these three sets; indeed, $\left(\frac{\pi}{3},\pi\right)\subset \A^-$.   Moreover, for any convex polygon $\Om$ other than an equilateral triangle, at most two interior angles of $\Om$ can lie outside of $\A^-$.    As noted in Definition~\ref{def: char poly}, every $a_{[\bxi]}$ is either an empty product (so equals one) or else is a product of an even number of factors $c(\alpha_j)$.   Thus in the case of convex polygons all of whose angles lie in $\A^-$, every $a_{[\bxi]}>0$ and cancellation cannot occur in the second bulleted item of Remark~\ref{ss:cpguide}. 

\begin{exa}\label{exa:equilateral} Let $\pc_n^-$ denote the set of all convex $n$-gons that have all angles in $\A^-$.   For any convex $n$-gon $\Om$ (whether or not in $\pc_n^-$), the maximum possible cosine frequency occurring in $P_{\Om}(t)-\cos(t)$ is $1-2\ell_{\min}(\Om)$, where $\lmin(\Om)$ is the smallest edge length of $\Om$.   For $\Om\in\pc_n^-$, this cosine frequency necessarily occurs.   Thus within $\pc_n^-$, the characteristic polynomial of $\Om$ determines $\lmin(\Om)$.   In particular, the characteristic polynomial determines which elements of $\pc_n^-$ are equilateral, since for any $n$-gon of perimeter one, we have $\lmin(\Om)\leq\frac{1}{n}$ with equality only if $\Om$ is equilateral.    Moreover, equilateral elements of $\pc_n^-$ are distinguishable by their characteristic polynomials from all elements of $\displaystyle\bigcup_{m=n+1}^\infty\,\pc_m^-$.   
\end{exa}

We can apply Example \ref{exa:equilateral} to the class of regular $n$-gons.

\begin{thm} \label{th:regular}
    Within the class of regular $n$-gons, each element is uniquely determined by its characteristic polynomial. 
\end{thm}

\begin{proof}
For $n\neq 3$, the regular $n$-gon lies in $\pc_n^-$. Thus by Example~\ref{exa:equilateral}, the characteristic polynomial distinguishes the regular $n$-gon from the regular $m$-gon when $n\neq m$ with both $n, m\geq 4$.   The equilateral triangle $T$ can be distinguished from the higher order regular polygons by the absence of cosine terms in $P_{T}(t) -\cos(t)$.    
\end{proof}

By Example~\ref{exa:equilateral}, regular $n$-gons are also distinguishable from those $m$-gons, $m> n$ that have all angles in $\A^-$ (including in particular all $m$-gons that have all angles $> \frac{\pi}{3}$).    We next compare regular $n$-gons to other $n$-gons. Applying Example~\ref{exa:equilateral} again, we see that they are distinguishable from non-equilateral $n$-gons that have all angles in $\A^-$.
For small $n$, we can say more:

\begin{thm}\label{th:square and pent}~For $n\in \{3,4\}$ the regular $n$-gon is uniquely determined within the set of all convex $n$-gons by its characteristic polynomial.  For $n=5$, any convex 5-gon with the same characteristic polynomial as a regular pentagon must be equilateral.
\end{thm}

\begin{proof} For $n=3$, this is Theorem~\ref{th:triangles123}(e) and for $n=4$, it is a special case of Theorem~\ref{thm:rectangle}.  So, let $R$ be the regular pentagon, and let $Q$ be another convex $5$-gon. The characteristic polynomial for a regular pentagon of perimeter equal to one is 
\begin{align*}
    P_R(t)&:=\cos(t)+5\cos^2\left(\frac{5\pi}{6}\right)\cos\left(\frac{3t}{5}\right)+\left[5\cos^2\left(\frac{5\pi}{6}\right)+5\cos^4\left(\frac{5\pi}{6}\right)\right]\cos\left(\frac{t}{5}\right)\\
    &-\sin^5\left(\frac{5\pi}{6}\right)\\
    &=\cos(t)+\frac{15}{4}\cos\left(\frac{3t}{5}\right)+\frac{105}{16}\cos\left(\frac{t}{5}\right)-\frac{1}{32}.
\end{align*}

Assume that $Q$, a pentagon with angles $\balpha= (\alpha_1,\alpha_2,\ldots, \alpha_5)$ and lengths $\bell=(\ell_1,\ell_2,\ldots,\ell_5)$, has the same characteristic polynomial. Since $\frac{105}{16}>6$, at least 7 choices of $[\bxi]$ must contribute non-trivially to the $\cos\left(\frac{t}{5}\right)$ term.  Similarly, at least $4$ choices of $\bxi$ must contribute non-trivially to the $\cos\left(\frac{3t}{5}\right)$ term. 
Thus at least eleven choices of $\bxi$ must contribute non-trivially to the cosine terms. In particular, since the presence of a single odd angle would result in the vanishing of $a_{[\bxi]}$ for eight distinct $[\bxi]$, \begin{align}
\label{s:only4terms}
 Q \mbox{ cannot have any odd angles, and all except at most four }[\bxi]\mbox{ must satisfy } |\bxi\cdot\bell|\in \left\{1,\frac{3}{5}, \frac{1}{5}\right\}.  
\end{align}

Let $[\bxi(i)]$ be the equivalence class that distinguishes edge $\ell_i$, so for example 
\[ [\bxi(2)]=\pm (1,-1,1,1,1), \]
i.e., $\bxi(i)$ has the $i^{th}$ component with a sign different from all other components.  Similarly, let $\bxi(i,j)$ denote the element that  distinguishes $\ell_i+\ell_j$. For example, 
\[ \bxi(2,4)= (1,-1,1,-1,1). \] 
Suppose  that $Q$ has at least one edge length that is not a multiple of $\frac{1}{5}$. Let \begin{align*}
    A&:= \left\{j: \ell_j \mbox{ is not a multiple of } \frac{1}{5}\right\},\\
    B&:= \{1,2,3,4,5\}\setminus A= \left\{k: \ell_k \mbox{ is  a multiple of } \frac{1}{5}\right\}.
\end{align*}
Let $a:= \# A >0$ and $b:= \# B= 5-a$. For $j\in A$ and $k\in B$, neither $|\bxi(j)\cdot\bell|$ nor $|\bxi (j,k)\cdot\bell|$ lies in $\left\{1,\frac{3}{5}, \frac{1}{5}\right\}$. Since $a+ab =a(6-a)>4$, this contradicts \eqref{s:only4terms}.  Therefore, all the edge lengths of $Q$ are multiples of $\frac{1}{5}$. Since 
\[ \sum_{j=1}^5 \ell_j =1, \] 
we have $\ell_j=\frac{1}{5} $ for  $j=1,2,3,4,5$.
\end{proof}

Next we investigate convex quadrilaterals that have the same characteristic polynomial as an equilateral triangle.

\begin{thm} \label{th:equilateral}  
Assume that a convex quadrilateral has the same characteristic polynomial as an equilateral triangle of perimeter one.  Then it must satisfy the following conditions:
\begin{enumerate} 
\item \label{th:equilaterali} the quadrilateral's side lengths satisfy:  $\ell_1+\ell_2=\ell_3+\ell_4=\frac 1 2$; 
\item \label{th:equilateralii} the quadrilateral has precisely two non-adjacent odd angles; and
\item \label{th:equilateraliii} the odd angles are of the form 
\[ \frac{\pi}{2k+1}, \quad \frac{\pi}{2j+1},\]
for positive integers $j$ and $k$ of different parities, and the other two angles are equal to each other.  Hence all four angles are rational. 
\end{enumerate}

If, moreover, a convex quadrilateral is Steklov isospectral to an equilateral triangle, then there are only finitely many values the angles may assume.  For each such value of $\balpha$, there is a continuous family of quadrilaterals with this data that satisfy (1) and have the same characteristic polynomial as the equilateral triangle.
\end{thm}

\begin{proof} 
The characteristic polynomial of an equilateral triangle with perimeter equal to one is 
\[ \cos(t) + 1.\]

Assume that $Q$, a quadrilateral with angles $\balpha = (\alpha_1, \alpha_2, \alpha_3,\alpha_4)$ and side lengths $\bell = (\ell_1, \ell_2, \ell_3,\ell_4),$ has the same characteristic polynomial.  We consider three cases, depending on the number of odd angles of $Q$, which is necessarily 0, 1, or 2 since $Q$ is convex.

In case $Q$ has no odd angles, we postpone the proof to the next section; see Remark~\ref{rem:quad no odd}.

 Now assume that one of the angles of $Q$ is odd; without loss of generality, let this angle be $\alpha_1$.  Then the characteristic polynomial of $Q$ is given by 
\beq P_Q (t) &=& \cos(t) + c(\alpha_3) c(\alpha_4) \cos (|\ell_1 + \ell_2+\ell_3 - \ell_4|t) + c(\alpha_2) c(\alpha_3) \cos (|\ell_1 + \ell_2-\ell_3 + \ell_4|t)\\ \nonumber
&& + \ c(\alpha_2) c(\alpha_4) \cos (|\ell_1 + \ell_2-\ell_3 - \ell_4|t) - s(\alpha_1) s(\alpha_2) s(\alpha_3)s(\alpha_4). \label{eq:charp_one_odd}
\eeq 
Since the frequencies in the second and third cosine terms cannot vanish, in order for the constant term in $P_Q$ to be equal to one, we must have $|\ell_1+\ell_2-\ell_3-\ell_4| = 0$ and $c(\al_2)c(\al_4)>0$.   
Moreover, the frequencies in the second and third cosine terms must be identical, and the coefficients must cause these terms to cancel.  We therefore have 
\[ c(\al_3) c(\al_4) + c(\al_2) c(\al_3) = 0, \]
so that $c(\al_3) = 0$ or $c(\al_4) = - c(\al_2)$. The latter option contradicts $c(\al_2)c(\al_4)>0$, and the former option implies that $\alpha_3$ is also odd.  Thus, if $Q$ has exactly one odd angle, it cannot have the same characteristic polynomial as an equilateral triangle.

Finally, we consider the case when $Q$ has two odd angles.  We saw above that labeling one of the odd angles as $\al_1$ leads to the other odd angle being $\al_3$, so that condition \eqref{th:equilateralii} of the theorem holds.  The characteristic polynomial becomes 
\begin{equation*} 
P_Q (t) = \cos(t) + 
 \ c(\alpha_2) c(\alpha_4) \cos (|\ell_1 + \ell_2-\ell_3 - \ell_4|t) \pm s(\alpha_2)s(\alpha_4). \label{eq:charp_one_odd}
\end{equation*} 
Again, we must have $c(\al_2)c(\al_4)>0$ and $|\ell_1+\ell_2-\ell_3-\ell_4| = 0$, so that condition \eqref{th:equilaterali} of the theorem holds. Setting 
\[ \alpha_1 = \frac \pi {2k+1} \ \ \text{and} \ \  \alpha_3 = \frac{\pi}{2j+1}\]
for some $j$ and $k$, we must have that 
\[ 
1 = c(\al_2) c(\al_4) \pm s(\al_2) s(\al_4) = \cos \left(\frac{\pi^2}{2\alpha_2} \mp \frac{\pi^2}{2 \alpha_4}\right),\]
where the upper sign is used if $\al_1$ and $\al_3$ have opposite parity, and the lower sign is used otherwise.  \pink

If $\al_1$ and $\al_3$ have the same parity, then $\cos \left(\frac{\pi^2}{2\alpha_2} + \frac{\pi^2}{2 \alpha_4}\right)=1$ and thus
\begin{equation}\label{eqn:al2al4}
\frac{\pi^2}{2\alpha_2} + \frac{\pi^2}{2 \alpha_4} = 2n\pi
\end{equation}
for some $n \in \Z^+$. Since $\al_1 + \al_3 \leq \frac{2\pi}{3}$, we have $\al_2, \al_4 \in (\frac{\pi}{3},\pi)$ and $\al_4 \geq \frac{4\pi}{3} - \al_2$.  Consider the function
\[
f(\al) = \frac{\pi^2}{2 \al} + \frac{\pi^2}{2(\frac{4\pi}{3} - \al)}.
\]
It is straightforward to check that the maximum of $f$ on the closed interval $[\frac{\pi}{3}, \pi]$ is $2\pi$ and is attained only at the endpoints.  Thus $2 \pi>f(\al_2) \geq \frac{\pi^2}{2\alpha_2} + \frac{\pi^2}{2 \alpha_4}$, contradicting \eqref{eqn:al2al4}.

Thus $\al_1$ and $\al_3$ must have opposite parity, we have $\cos \left(\frac{\pi^2}{2\alpha_2} - \frac{\pi^2}{2 \alpha_4}\right)= 1$, and
\begin{equation}
\frac{\pi^2}{2 \alpha_2} - \frac{\pi^2}{2 \alpha_4} = 2m \pi
\end{equation}
for some $m \in \Z$. We again have $\al_2, \al_4 \in (\frac{\pi}{3},\pi)$, so that $\frac{\pi^2}{2\al_2}, \frac{\pi^2}{2\al_4} \in (\frac{\pi}{2}, \frac{3\pi}{2}).$  Thus $m=0$, $\al_2 = \al_4$, and condition (3) of the theorem holds.

By \cite[Corollary 3.12]{steklov1},   convex quadrilaterals that are Steklov isospectral to a given equilateral triangle have interior angles that are uniformly bounded from below by a positive constant.  Consequently, there are at most finitely many possible values of $j$ and $k$ and therewith at most finitely many values for the angles.  
 By \cite[Lemma 5.8]{steklov1}, for each choice of angle values, there is a continuous family of quadrilaterals that have this same angle and length data, and thus necessarily the same characteristic polynomial as the equilateral triangle.
\end{proof} 

\begin{remark} \label{rem:impossiblequads}
While we were in the process of preparing these results for publication, the first named author supervised an undergraduate research project by Dante Mancino at Bucknell University.  Mancino proved that no convex quadrilateral can simultaneously satisfy all three conditions in Theorem \ref{th:equilateral} \cite{ManPrivateComm}. Thus no convex quadrilateral can have the same characteristic polynomial as an equilateral triangle.  
\end{remark}

\section{Smoothly bounded domains}\label{s:smooth} 

We now turn to the question of whether a convex polygonal domain can have the same characteristic polynomial as a simply-connected smoothly bounded domain.    By Proposition~\ref{prop:vertices}, admissible curvilinear $n$-gons with $n\geq 1$ are always distinguishable from simply-connected smoothly bounded domains by their characteristic polynomials. 
For convex $n$-gons that are not necessarily admissible, we will begin by showing that triangles and quadrilaterals are distinguishable from simply-connected smoothly bounded domains.  We will then address the case of $n$-gons with $n\geq 5$.  

\subsection{Triangles and quadrilaterals}\label{subsec:triquad} We showed in Theorem~\ref{th:triangles123} that the equilateral triangle is uniquely determined among all triangles by its characteristic polynomial, and showed in Theorem~\ref{th:equilateral} that at most finitely many convex quadrilaterals have the same characteristic polynomial as the equilateral triangle. The quadrilaterals necessarily satisfied certain conditions, but by Remark \ref{rem:impossiblequads}  no quadrilateral can simultaneously satisfy these conditions, so the equilateral triangle is uniquely determined among all triangles and quadrilaterals by its characteristic polynomial. We now show that the characteristic polynomial of a triangle or a quadrilateral is never equal to that of a smoothly bounded domain. 

\begin{thm}\label{res:quadrilaterals} 
No triangle or quadrilateral has the same characteristic polynomial as a smoothly bounded domain. 
\end{thm}

 \begin{proof}

 By \eqref{eq:triangle_charpoly}, only triangles with two or three odd angles have a single cosine term $\cos(t)$ in their characteristic polynomial. The constant term $s(\alpha_1)s(\alpha_2)s(\alpha_3)$ has magnitude one if and only if $T$ has three odd angles and is necessarily equilateral.  In that case, the sign of the constant term distinguishes the characteristic polynomial of the equilateral triangle from that of the smoothly bounded domains.  Thus, no triangle has the same characteristic polynomial as a smoothly bounded domain.
 
Let $Q$ be a quadrilateral of perimeter one.  We need to show that $P_Q(t) \not\equiv \cos(t)-1$.  We suppose the contrary.  Denote the angles of $Q$ by $\alpha_1,\dots, \al_4$.   We will consider separately the three cases that the number of odd angles is zero, one or two.    We first observe that in all cases, $|s(\al_1)s(\al_2)s(\al_3)s(\al_4)| <1$; consequently in the notation of Definition~\ref{def: char poly}, 
\beq\label{eq:exist} \exists\, \bxi\in\{\pm 1\}^4 \mbox{\,\,such that \,} \bxi\cdot\bell=0 \mbox{\,\,and\,\,} a_{\bxi} <0. \eeq
  Necessarily, $\bxi$ will have two positive entries and two negative ones.

 We first consider the case that $Q$ has no odd angles, so $a_{\bxi}\neq 0$ for every $\bxi$.  Denote the edge lengths in increasing (not necessarily cyclic) order as 
\[\ell\leq \ell'\leq\ell''\leq \ell'''.\]
By the observation in the previous paragraph, we must have 
\beq\label{eq: sum cond}\ell+\ell''' =\ell'+\ell''.\eeq
Moreover, our assumption that $P_Q(t)=\cos(t)-1$, which does not include a cosine term of frequency $-\ell +\ell'+\ell''+\ell'''$, implies that at least two edges must realize the minimal length $\ell$, i.e., $\ell=\ell'$.  Equation~\eqref{eq: sum cond} then yields $\ell''=\ell'''$, so $Q$ has two pairs of equal edge lengths.  Thus $Q$ is either a parallelogram or a kite with no odd angles.   A glance at Equations~\eqref{eq:char parallelogram} and \eqref{eq:char_kite} then shows that the characteristic polynomial cannot equal $\cos(t)-1$, completing this case.

 Next suppose that $Q$ has exactly one odd angle, say $\alpha_1$.  Using our usual cyclic notation for edge lengths, we have 
 \beq P_Q(t)= \cos(t)&+&c(\al_2)c(\al_3)\cos(|\ell_1+\ell_2-\ell_3+\ell_4|t) + c(\al_3)c(\al_4)\cos(|\ell_1+\ell_2+\ell_3-\ell_4|t) \nn \\&+&c(\al_2)c(\al_4)\cos(|\ell_1+\ell_2-\ell_3-\ell_4|t) \,\pm s(\al_2)s(\al_3)s(\al_4),\label{eq: Q 1 odd}\eeq
 where the sign in the last term depends on the odd angle $\alpha_1$.   Comparing this with Equation~\eqref{eq:exist}, we must have 
 \beq\label{eq:12=34} \ell_1+\ell_2=\ell_3+\ell_4\, \left(=\frac{1}{2}\right)\eeq 
 and
  $c(\al_2)c(\al_4)<0.$   The other two cosine frequencies must cancel, so 
 \beq\label{eq:ellell4} \ell_3=\ell_4=\frac{1}{4}\eeq and $c(\al_3)[c(\al_2)+c(\al_4)]=0$.  Note that $c(\al_3)\neq 0$ since we assumed that $\al_1$ is the only odd angle.  Thus $c(\al_4)=-c(\al_2)$.  This implies that $s(\al_4)=\pm s(\al_2)$.   Equation~\eqref{eq: Q 1 odd} thus simplifies to 
 \beq P_Q(t)&=&\cos(t) -c(\al_2)^2 \pm s(\al_2)^2s(\al_3) \nn \\&=& \cos(t) -1+ s(\al_2)^2[1\pm s(\al_3)].\eeq
 Since $\al_3$ is not odd, $1\pm s(\al_3)\neq 0$.  Thus our assumption that $P_Q(t)=\cos(t)-1$ implies that $s(\al_2)=0$, i.e., $\al_2$ is even.   $\al_4$ is even as well but distinct from $\al_2$ since $c(\al_4)=-c(\al_2)$.   Given that $\al_1$ is odd and that necessarily $\al_1+\al_2+\al_4>\pi$, the only possibility is that $\al_1=\frac{\pi}{3}$ and $\{\al_2,\al_4\}=\{\frac{\pi}{2}, \frac{\pi}{4}\}$ and then $\al_3=\frac{11\pi}{12}$.   This together with Equation~\eqref{eq:ellell4} shows that the line segment $s$ joining vertices $\al_2$ and $\al_4$ divides $Q$ into two triangles:  (i) an isosceles triangle with angles $\frac{\pi}{24}, \frac{11\pi}{12}, \frac{\pi}{24}$, two edges of length $\frac{1}{4}$, and edge $s$ of length $\frac{1}{2}\cos(\pi/24)$;  (ii) a triangle $T$ with angles $\frac \pi 3$, $\frac{5\pi}{24}$ and $\frac{11\pi}{24}$ and sides of length $\ell_1$, $\ell_2$ and $s=\frac 1 2 \cos(\pi/24)$.  
The Law of Sines gives that either $\ell_1$ or $\ell_2$, whichever is opposite the angle measuring $\frac{11\pi}{24}$, measures $\frac 1 2 \frac{\cos(\pi/24) \sin(11\pi/24)}{\sin(\pi/3)} > \frac 1 2$, contradicting Equation \eqref{eq:12=34}. This completes the proof in the case that $Q$ has exactly one odd angle.

Finally suppose that $Q$ has two odd angles.   If the odd angles are adjacent, say $\al_1$ and $\al_2$, then we have
$P_Q(t)=\cos(t) +c(\al_3)c(\al_4)\cos(|\ell_1+\ell_2+\ell_3-\ell_4|t) \pm s(\al_3)s(\al_4)$, which does not coincide with $ \cos(t)-1.$  Thus we may assume that the odd angles are non-adjacent, say $\alpha_1$ and $\alpha_3$.  We then have
$P_Q(t)=\cos(t) +c(\alpha_2)c(\al_4)\cos(|\ell_1+\ell_2-\ell_3-\ell_4|t) \pm s(\al_2)s(\al_4).$ Since $\al_1+\al_3\leq \frac{2\pi}{3}$, both $\al_2$ and $\al_4$ must lie in $\left(\frac{\pi}{3},\pi\right)$ and thus $c(\al_2)c(\al_4)>0$.  Consequently, Condition~\eqref{eq:exist} fails to hold.  This contradiction completes the proof.
 \end{proof}

 \begin{remark}\label{rem:quad no odd}
The proof that a quadrilateral $Q$ with no odd angles cannot have the same characteristic polynomial as a smooth domain is easily modified to show that $Q$ cannot have the same characteristic polynomial as an equilateral triangle. The characteristic polynomial of the equilateral triangle agrees with that of a smooth domain except for the sign of the constant term.  Thus in Equation~\eqref{eq:exist}, we need to replace $a_{\bxi} <0$ by $a_{\bxi} >0$.  The rest of the argument in the case that $Q$ has no odd angles does not use the constant term and thus goes through without change.
     \end{remark}

\subsection{Higher $n$-gons}\label{subsec:ngons}

In this subsection, we address the question of whether an $n$-gon $\Om$ with $n\geq 5$ can have the same characteristic polynomial as a simply-connected smoothly bounded domain.    We will not resolve the question but will obtain strong necessary conditions.  When a polygon $\Omega$ has odd interior angles, we introduce a curvilinear polygon, $\omred$ ($\Omega$ reduced), obtained by removing the odd vertices of $\Omega$.

\begin{defn}\cite[Defn. 5.1; Notation \& Remarks 2.12]{steklov1} \label{def: omega reduced} 
Let $\Om$ be a convex $n$-gon.
 \begin{enumerate} 
  \item[(a)]
 Let $k$ be the number of odd interior angles in $\Om$.  If $k=0$, set $\omred:=\Om$.  If $k=1$ or 2, let $\omred$ be a curvilinear $(n-k)$-gon obtained by ``removing'' the vertices where the odd angles occur.   More precisely, if $\alpha_j$ is an odd angle and $\ell_j$ and $\ell_{j+1}$ are the lengths of the two edges that meet at the vertex with angle $\alpha_j$, then replace the two edges by a single smooth curve of length $\ell_j+\ell_{j+1}$, being careful not to affect the adjacent vertex angles $\alpha_{j-1}$ and $\alpha_{j+1}$.   If there are two odd angles, repeat the process.   In particular, if odd angles $\alpha_{j-1}$ and $\alpha_j$ occur at adjacent vertices of $\Om$, then the three edges incident to these two vertices are replaced by a single smooth curve of length $\ell_{j-1}+\ell_j+\ell_{j+1}$.  The only convex polygons with more than two odd angles are equilateral triangles.   In this case, $\omred$ is a smooth simply-connected domain, and the characteristic polynomial of $\omred$ is defined as in Equation \eqref{eq:cp_smooth_domain}. We refer to $\omred$ as the reduced curvilinear $(n-k)$-gon associated with $\Om$.   
   
 \item[(b)] We say that a convex $n$-gon is \emph{weakly edge-admissible} if the edge lengths of $\omred$ are incommensurable over $\{-1,0,1\}$. Observe that incommensurability of the edge lengths of $\Om$ over $\{-1,0,1\}$ implies that $\Om$ is weakly edge-admissible.

\item[(c)]  For a curvilinear $n$-gon $\Om$ with interior angles $\alpha_1,\dots, \alpha_n$, we write $\pmca(\Om)=(|c(\alpha_1)|,\dots, |c(\alpha_n)|)$ with $c(\alpha)$ defined in Equation \eqref{eq:c(alpha)}.
\end{enumerate}
\end{defn}

Since there is a choice of the smooth curves that replace the edges that meet at an odd angle, $\omred$ is not well-defined; this will not matter in our use of $\omred$.  We use only the lengths of the smooth curves and the fact that, by virtue of not being straight line segments, we can distinguish them among the edges of $\omred$.  The utility of $\omred$ lies in the relationship between the characteristic polynomials of $\Om$ and $\omred$.  

\begin{lemma}\cite[Lemma 5.3]{steklov1}\label{om vs omred} We use the notation of Definition~\ref{def: omega reduced}.   Let $\Om$ be a weakly edge-admissible convex $n$-gon.  Let $k$ be the number of odd interior angles in $\Om$.   Then:  
\begin{enumerate}
\item[(a)] $\omred$ is either an admissible curvilinear $(n-k)$-gon or a domain with smooth boundary if $n=k=3$;
\item[(b)] 
The characteristic polynomials of $\Om$ and $\omred$ are identical except possibly for a change in the sign of the constant term.  The sign will depend on the parity of the odd angles in the sense of Definition~\ref{def:exceptional}.
\item[(c)] If $\Om$ is not an equilateral triangle, the characteristic polynomial of $\Om$
determines $\pmca(\omred)$ and $\bell(\omred)$ up to possible permutations of the entries. Moreover, unless $\Om$ has more than one even angle, the characteristic polynomial of $\Om$ determines $\pmca(\omred)$ and $\bell(\omred)$ uniquely (modulo the choice of boundary orientation and cyclic labelling).  

\end{enumerate}
\end{lemma}

Equilateral triangles are excluded in part (c) only because we have not defined $\pmca(\omred)$ when $\omred$ has smooth boundary.  As we investigate the possibility that a polygonal domain $\Omega$  has the same characteristic polynomial as a smoothly bounded domain, it is convenient to keep track of those side lengths of $\Omega$ that cannot be expressed as the sum of two or more smaller side lengths.  For this we define $\mathcal L(\Omega)$ and $\mathcal L^*(\Omega)$ below.

\begin{nota}\label{nota:edgelengths}
Given a curvilinear $n$-gon $\Om$, let $\Lc(\Om)$ denote the collection of all edge lengths of $\Om$ repeated with multiplicity. We will generally write the entries of $\Lc(\Om)$ in the order of increasing size as opposed to their cyclical order of occurrence in $\Om$.   Let $\Lc^*(\Om)$ denote the set of elements -- without multiplicity -- of $\Lc(\Om)$ that cannot be expressed as the sum of two or more smaller elements of $\Lc(\Om)$. For example, if $\Om$ is a 6-gon with $\Lc(\Om)=\left(\frac{1}{16},\frac{1}{8},\frac{1}{8},\frac{3}{16},\frac{1}{4},\frac{1}{4}\right)$, then $\Lc^*(\Om)=\left\{\frac{1}{16}, \frac{1}{8}\right\}$.
 \end{nota}

To obtain our first collection of necessary conditions that an $n$-gon $\Om$ of perimeter one must satisfy in order that $P_{\Om}(t)=\cos(t)-1$, we use the fact that no cosine terms appear in $P_{\Om}$ other than $\cos(t)$.

\begin{lemma}\label{lem:angle in +} Let $\Om$ be a convex $n$-gon of perimeter one with the same characteristic polynomial as a smoothly bounded simply-connected domain.   Then at least one angle of $\Om$ lies in $\A^+$ (see Notation \ref{nota:A}), and $\Om$ has at most one odd angle.
\end{lemma}

\begin{proof}
 By Theorem \ref{res:quadrilaterals}, $n \geq 5$, so at most two angles of $\Om$ are odd.   It follows that there exist choices of $[\bxi]$ such that $0<\bxi\cdot\bell <1$ and $a_{[\bxi]}\neq 0$.    The fact that there is no cosine term of frequency $\bxi\cdot\bell$ in $P_{\Om}$ means that cancellation occurs.  Thus there exist $\bxi'$ such that $a_{[\bxi]}$ and $a_{[\bxi']}$ have opposite signs.  Since each of $a_{[\bxi]}$ and $a_{[\bxi']}$ is a product of an even number of factors $c(\al_i)$, the first statement of the lemma follows. Next, since at most two angles of $\Om$ can lie in $(0,\frac{\pi}{3}]$ and at least one is in $\A^+$, there can be at most one odd angle.
\end{proof}

\begin{prop} \label{prop:n_smoothbdd2}
Let $\Om$ be a convex $n$-gon of perimeter one with the same characteristic polynomial as a smoothly bounded simply-connected domain.  First suppose that $\Om$ has no odd angles.   Then:

\begin{enumerate}
\item[(a)] Every $\ell\in \Lc^*(\Om)$ occurs with multiplicity at least two in $\Lc(\Om)$; in particular, $\Om$ cannot have a unique shortest edge.  Moreover, there exist edges $E$ and $E'$ of $\Om$ of length $\ell$ such that the angle at exactly one endpoint of $E$ lies in $\A^+$, whereas either both or neither of the angles at the endpoints of $E'$ lies in $\A^+$.   

\item[(b)] Let $\ell\in\Lc^*(\Om)$.   Then either $\ell$ occurs with multiplicity $m(\ell)=n$ in $\Lc(\Om)$ (so $\Om$ is equilateral), or else there exists $k$ with $2\leq k \leq m(\ell)$ such that $k\ell$ is a sum of \emph{one or more} elements of $\Lc(\Om)$ distinct from $\ell$.   \end{enumerate}

Next suppose that $\Om$ has one odd angle.  Then (a) and (b) hold with $\Lc(\Om)$, $\Lc^*(\Om)$ replaced by $\Lc(\omred)$, $\Lc^*(\omred)$  and $n$ by $n-1$.
\end{prop}

\begin{proof}
By Theorem \ref{res:quadrilaterals}, $n \geq 5$.
(a) Suppose that $\Om$ has no odd angles, so $a_{[\bxi]}\neq 0$ for all $[\bxi]$.   Let $\ell\in\Lc^*(\Om)$.   Each edge $E$ of length $\ell$ contributes a term $c(\al)c(\al')\cos(|1-2\ell|t)$ to $P_{\Om}$ where $\al$ and $\al'$ are the angles at the endpoints of $E$.   Moreover, the definition of $\Lc^*$ implies that the only terms of this form arise from edges of length $\ell$.    Note that $\ell <\frac 1 2$ by the triangle inequality, so $1-2\ell\neq 0$.  The hypothesis that $\Om$ has the same characteristic polynomial as a smooth domain implies that all the terms of this form must cancel, and statement (a) follows.

(b) Suppose the multiplicity $m:=m(\ell)$ is less than $n$.  We need to show that there exists $k\leq m$ such that $k\ell$ is a sum of \emph{one or more} elements of $\Lc$ distinct from $\ell$. Let $E_{i_1}, \dots, E_{i_m}$ be the edges of length $\ell$.    Define $\bxi=(\xi_1,\dots, \xi_n)$ so that $\xi_{j}=-1$ for $j \in \{i_1,\dots, i_m\}$ while all other $\xi_j$ equal $1$.  Then $\cos(|\bxi\cdot \bell |t)= \cos(|1-2m\ell |t)$.   If $m\ell=\frac{1}{2}$, then the sum of all the edge lengths distinct from $\ell$ must equal $\frac 1 2$, so the desired conclusion holds with $k=m$.   Otherwise, cancellation of the terms of form $\cos(|1-2m\ell|t)$ in the characteristic polynomial requires that there exists some other subset $S$ of edges of $\Om$ for which the sum of the edge lengths is also $m\ell$.   Let $q$ be the number of edges in $S$ that have length $\ell$.   Necessarily $q<m$, and the conclusion of (b) holds with $k=m-q$.    
 
Next suppose that $\Om$ has one odd angle.  By Lemma \ref{om vs omred}(b), the characteristic polynomial of $\omred$ agrees with that of $\Om$ except possibly for the sign of the constant term.   The argument above does not use the constant term and thus goes through for $\Om$ replaced by $\omred$ once we observe the following:   In the proof of (a) we used the fact that no edge of $\Om$ can have length $\frac 1 2$.   This is no longer true for $\omred$ since the two edges of $\Om$ adjacent to the odd angle are combined to form a single edge of $\omred$.   However, if this edge, denoted $E$, has length $\frac 1 2$, then the sum of the remaining edges of $\omred$ also equals $\frac 1 2$.  Thus the edge length $\frac 1 2$ does not belong to $\Lc^*(\omred)$.
\end{proof}

\begin{remark}\label{rem:shortest}  Suppose that $\Om$ satisfies the hypotheses of Proposition~\ref{prop:n_smoothbdd2}.    If $\Om$ has no odd angles, statement (a) implies that the shortest edge length $\ell_0$ occurs with multiplicity at least two.  Moreover, using (b), the multiplicity will be at least three unless some edge of $\Om$ has length $2\ell_0$.   Indeed, minimality of the length $\ell_0$ implies that the only way to express $2\ell_0$ as a sum of edge lengths is as $\ell_0 +\ell_0$.     Consequently, the integer $k$ in (b) must be at least 3 and thus $m(\ell_0)\geq 3$.   If $\Om$ has an odd angle, the same conclusion holds with $\Om$ replaced by $\omred$.   

\end{remark}

\begin{cor}\label{cor:lc* size} Let $\Om$ be a convex $n$-gon with the same characteristic polynomial as a smoothly bounded simply-connected domain.   In the notation of Proposition~\ref{prop:n_smoothbdd2}, we have
\[|\Lc^*(\Om)|\leq 4.\]
Moreover, we have a smaller bound in the following cases:

\begin{equation*}
\begin{cases}
|\Lc^*(\Om)|\leq 2 \,\,\rm{if}\,\, n=5, 6 \,\,\rm{and}\,\Om\,\rm{has\, no \,odd \,angles,}\\
|\Lc^*(\Om)|\leq 3 \,\,\rm{if}\,\, n=7, 8\,\,{and}\,\Om\,\rm{has\, no \,odd \,angles,}\\
|\Lc^*(\Om)|\leq 3 \,\,\rm{if}\,\, n=5 \,\,{and}\,\,\,\Om\,\rm{\, has\, 1 \,odd \,angle}.
\end{cases}
\end{equation*}

\end{cor}

We emphasize that we are counting the elements of $\Lc^*(\Om)$, not of $\Lc^*(\omred)$ even when $\Om$ has an odd angle.  However, the proof below also yields a bound on the latter, which is sometimes smaller than the bound on the former.

\begin{proof}
By Theorem \ref{res:quadrilaterals}, $n \geq 5$. First suppose that $\Om$ has no odd angles.     At most two interior angles lie in $\A^+$, since such angles are less than $\frac \pi 3$.  Thus at most four edges are adjacent to a vertex with angle in $\A^+$, so the bound $|\Lc^*(\Om)|\leq 4$ follows from Proposition~\ref{prop:n_smoothbdd2}(a).   

To obtain the smaller bounds when $5\leq n\leq 8$, the shortest edge length $\ell_0$ accounts for one element of $\Lc^*$, and Remark~\ref{rem:shortest} tells us that there are at most $n-3$ additional elements. Applying Proposition~\ref{prop:n_smoothbdd2}(b), we thus obtain the bound $|\Lc^*(\Om)|\leq 1 +\frac{n-3}{2}$, which outperforms the bound $|\Lc^*(\Om)|\leq 4$ when $n\leq 8$.   

Next suppose that $\Om$ has one odd angle.  We first bound $|\Lc^*(\omred)|$.  Since $\Om$ has an odd angle, at most one interior angle of $\Om$ -- and thus at most one interior angle of $\omred$ -- can lie in $\A^+$. Also, since $\omred$ has only $n-1$ edges, by an argument analogous to the one above, we obtain
\[|\Lc^*(\omred)|\leq \min\left\{2, 1+\frac{n-4}{2}\right\}.\]
In particular $|\Lc^*(\omred)|=1$ when $n=5$ and $|\Lc^*(\omred)|\leq 2$ for larger $n$.

Finally, the collections of edge lengths $\Lc(\Om)$ and $\Lc(\omred)$ differ only in that the lengths of the two edges of $\Om$ adjacent to the odd angle are not included in $\Lc(\omred)$, only their sum.    It follows easily that $|\Lc^*(\Om)|\leq |\Lc^*(\omred)|+2$.  This completes the proof.
\end{proof}

We next use the fact that the constant term in $P_{\Om}$ must equal $-1$.  

\begin{prop} \label{prop:n_smoothbdd1}
Let $\Om$ be a convex $n$-gon of perimeter one.  Suppose that $\Om$ has the same characteristic polynomial as a smoothly bounded simply-connected domain. 
Then either:
\begin{itemize}
\item[(i)] $\Om$ has two even angles of opposite parity that separate $\pa\Om$ into two components of length $\frac 1 2$, 
\item[(ii)] or there exist at least two equivalence classes $[\bxi], [\bxi']$ of elements of $\{\pm 1\}^n$ such that $\bxi\cdot\bell =0=\bxi'\cdot\bell$ and such that $a_{[\xi]}, a_{[\xi']}$ are both negative.
\end{itemize}

\end{prop}
We note that the two possibilities are not mutually exclusive. 

\begin{proof} In view of Theorem~\ref{res:quadrilaterals}, we may assume that $n\geq 5$.   
The constant term in $P_{\Om}$ is given by 
\beq\label{const term}
-1= \sum_{\{[\bxi]: \,\bxi\cdot\bell =0\}}a_{[\bxi]} - \prod_{i=1}^n\,s(\al_i).
\eeq
 If $\prod_{i=1}^n\,s(\al_i)<0$, then condition (ii) must hold since $|a_{[\bxi]}|\leq 1$ for all $\bxi$.   If $\prod_{i=1}^n\,s(\al_i)=0$, i.e., if $\Om$ has at least one even angle, then either (ii) holds or else there exists $\bxi$ such that $\bxi\cdot\bell=0$ and $a_{[\bxi]} =-1$.   In the latter case, each of the cosine factors $c(\al_j)$ that occur in $a_{[\bxi]}$ have magnitude one and thus correspond to even angles $\al_j$.    Since $\Om$ can't have four even angles, there must be exactly two even angles contributing to $a_{[\bxi]}$ and they must have opposite parity.   The condition $\bxi\cdot\bell=0$ implies that these two even angles separate $\pa\Om$ into two components of length $\frac 1 2$.

We are left with the case $\prod_{i=1}^n\,s(\al_i)>0$.      Suppose there is only one equivalence class $[\bxi]$ such that $\bxi\cdot\bell =0$ and $a_{[\bxi]}<0$.   Writing 
\[a_{[\bxi]}= c(\al_{i_1})\dots c(\al_{i_k})\]
for some $1\leq i_1\dots <i_k$ with $k\geq 2$, 
we have
\[ |a_{[\bxi]}|\leq |c(\al_{i_1})c(\al_{i_2})|.\]
Since $\Om$ has at most one odd angle by Lemma \ref{lem:angle in +}, all except at most one of the factors $s(\al_j)$ satisfy $|s(\al_j)|<1$.    Thus
\[ \left|\prod_{i=1}^n\,s(\al_i)\right| < |s(\al_{i_1})s(\al_{i_2})|. \]
Consequently, Equation~\eqref{const term} yields
\[1 = |a_{[\bxi]} - \prod_{i=1}^n\,s(\al_i)| < |c(\al_{i_1})c(\al_{i_2})| + |s(\al_{i_1})s(\al_{i_2})|. \]
Thus for a suitable choice of $\pm$, we have 
\[1 < \left |\cos\left(\frac{\pi^2}{2\al_{i_1}}\pm \frac{\pi^2}{2\al_{i_2}}\right) \right|,\]
a contradiction.   
\end{proof}

\begin{rem}\label{rem: sum half} Proposition~\ref{prop:n_smoothbdd1} implies that one can divide the edge lengths $\Lc(\Om)$ into two disjoint subcollections, $\Lc(\Om)= \Lc(\Om_1)\sqcup \Lc(\Om_2)$ in such a way that the sum of the lengths in each subcollection $\Lc(\Om_j)$ is exactly $\frac 1 2$.  
\end{rem}

\begin{cor}\label{cor:equilateral}
Convex equilateral polygons with an odd number of sides never have the same characteristic polynomial as a simply-connected smoothly bounded domain. 
\end{cor}

\begin{cor}\label{cor:not 1/2} If a convex $n$-gon $\Om$ of perimeter one with one odd angle has the same characteristic polynomial as a smoothly bounded simply-connected domain, then the sum of the lengths of the two edges that meet at the odd angle is strictly less than $\frac 1 2$.   

\end{cor}
\begin{proof} By Theorem \ref{res:quadrilaterals}, such an $n$-gon has $n \geq 5$. Assume $\Om$ satisfies the hypotheses.
Denote the odd angle by $\alpha_1$ and the corresponding vertex by $v_1$, so the lengths of the adjacent edges are denoted $\ell_1$ and $\ell_2$.  With our usual cyclic notation, edge $\ell_1$ has endpoints $v_1$ and $v_n$, and $\ell_2$ has endpoints $v_1$ and  $v_2$.  
The fact that $\alpha_1$ is odd implies that 
\[\xi_1=\xi_2\rm{\,\,for \,\,every\,} \bxi\in \{\pm 1\}^n\,\,\text{such\,\,that}\,\,a_{[\bxi]} \neq 0.\]   

If  $\ell_1 +\ell_2>\frac 1 2$, it follows that there does not exist $\bxi$ that simultaneously satisfies $a_{[\bxi]} \neq 0$ and $\bxi\cdot\bell =0$, contradicting Proposition~\ref{prop:n_smoothbdd1}. 
Next suppose that $\ell_1+\ell_2 =\frac 1 2$.  
Then $[(-1,-1, 1, 1,\dots, 1)]$ is the unique equivalence class $[\bxi]$ satisfying $a_{[\bxi]} \neq 0$ and $\bxi\cdot\bell =0$.  Proposition~\ref{prop:n_smoothbdd1} thus implies that $\al_n$ and $\al_2$ are even angles of opposite parity, say $\al_n=\frac{\pi}{4m}$ and $\alpha_2=\frac{\pi}{4p -2}$ for some $m,p \in \Z^+$.    The line segment from $v_n$ to $v_2$ decomposes $\Om$ into a triangle $T$ with vertices $v_n,v_1, v_2$ and a convex $(n-1)$-gon $\Om'$ with vertices $v_2, v_3, \dots v_n$.  Let $s$ denote the length of this line segment, so $T$ has edge lengths $\ell_1$, $\ell_2$ and $s$.   Denote the interior angles of $T$ by $\beta_1 \,(=\al_1)$, $\beta_2$ and $\beta_n$.   We have $\beta_2<\al_2\leq \frac{\pi}{2}$, and $\beta_n<\al_n\leq \frac{\pi}{4}$.  Since $\alpha_1$ is odd and $T$ has angle sum $\pi$, we must have $\alpha_1=\frac{\pi}{3}$, $\al_2=\frac{\pi}{2}$ and $\al_n=\frac{\pi}{4}$.    The fact that $\beta_2<\frac{\pi}{2}$ also implies that $\beta_n>\frac{\pi}{6}$.   

We find an upper bound on $s$ using that $\beta_n \in (\frac{\pi}{6}, \frac{\pi}{4}), \beta_2 \in (\frac{5\pi}{12}, \frac{\pi}{2})$, and $\ell_1 + \ell_2 = \frac{1}{2}$.  Given our labeling conventions and these angle restrictions, we know that $\ell_1 > \ell_2$, so we have $\ell_2 \in (0, \frac{1}{4}).$  The Law of Sines gives $s = \frac{\sqrt{3} \ell_2}{2 \sin \beta_n}$.  Defining $s = f(\ell, \beta) = \frac{\sqrt{3} \ell}{2 \sin \beta}$, we find that $\frac{\partial f}{\partial \ell} >0$ and $\frac{\partial f}{\partial \beta} < 0$ on the interior of our domain, so there are no interior critical points for $s$.  Checking the boundary, we see that the maximum value of $s=\frac{\sqrt{3}}{4}\approx 0.43$ occurs when $\ell=\frac{1}{4}$ and $\beta = \frac{\pi}{6}$.  Although this combination of values does not correspond to an actual triangle satisfying our constraints, we have found a sufficient upper bound on $s$. 

Now consider the $(n-1)$-gon $\Om'$.   The length of its ``base'' joining vertices $v_n$ and $v_2$ is $s$.   Let $\al_n'=\al_n-\beta_n$ and $\al_2'=\al_2-\beta_2$ denote its interior angles at these vertices.   Since $\beta_n+\beta_2=\pi-\al_1=\frac{2\pi}{3}$, and $\al_2+\al_n =\frac{3\pi}{4}$, we have 
\[\al_n'+\al_2' =\frac{3\pi}{4}-\frac{2\pi}{3} =\frac{\pi}{12}.\]
In particular, $\Om'$ is an extremely thin $(n-1)$-gon.  The sum of the remaining $n-3$ angles is $(n-3)\pi -\frac{\pi}{12}$.     On the other hand, the lengths of the edges other than the ``base'' edge sum to $\ell_3+\ell_4+\dots+\ell_n=1-(\ell_1+\ell_2)=\frac{1}{2}$.   
The base $s$ must then satisfy $s>\frac{1}{2}\cos(\pi/12)$, which is larger than $0.43$.  This completes the proof. 
\end{proof}

\begin{exa} \label{exa:penta}    
Let $\Om$ be a pentagon of perimeter one with the same characteristic polynomial as a simply-connected domain with smooth boundary.    
Assume that $\Om$ has no odd angles.  
Using Proposition~\ref{prop:n_smoothbdd2}, Corollary~\ref{cor:lc* size}, and Remarks~\ref{rem:shortest} and \ref{rem: sum half}, along with the fact that no individual edge can have length $\frac 1 2$, one verifies that the list below exhausts all the possibilities for $\mathcal L(\Omega)$.  We leave the case that $|\Lc^*(\Omega)|=1$ to the reader and give a brief explanation for the case $|\Lc^*(\Om)|=2$.

\smallskip
\noindent\underline{\rm{Possibilities for which $|\Lc^*(\Om)|=1$ in the notation of the corollary:}}
\begin{itemize}
\item $\left(\frac{1}{6}, \frac{1}{6}, \frac{1}{6}, \frac{1}{6}, \frac{1}{3}\right)$
\item $\left(\frac{1}{8}, \frac{1}{8}, \frac{1}{8}, \frac{2}{8}, \frac{3}{8}\right)$
\item $\left(\frac{1}{10}, \frac{1}{10}, \frac{1}{10}, \frac{3}{10}, \frac{4}{10}\right)$
\item $\left(\frac{1}{8}, \frac{1}{8}, \frac{2}{8}, \frac{2}{8}, \frac{2}{8}\right)$
\item $\left(\frac{1}{10}, \frac{1}{10}, \frac{2}{10}, \frac{2}{10}, \frac{4}{10}\right)$
\item $\left(\frac{1}{10}, \frac{1}{10}, \frac{2}{10}, \frac{3}{10}, \frac{3}{10}\right)$
\item $\left(\frac{1}{12}, \frac{1}{12}, \frac{2}{12}, \frac{3}{12}, \frac{5}{12}\right)$
\item $\left( \frac{1}{12}, \frac{1}{12}, \frac{2}{12}, \frac{4}{12}, \frac{4}{12} \right)$
\item $\left( \frac{1}{14}, \frac{1}{14}, \frac{1}{7}, \frac{2}{7}, \frac{3}{7} \right)$
\end{itemize}

\smallskip
\noindent\underline{\rm{Possibilities for which $|\Lc^*(\Om)|=2$ in the notation of the corollary:}}
\begin{itemize}
\item $\left(\frac{1}{6}, \frac{1}{6}, \frac{1}{6}, \frac{1}{4}, \frac{1}{4}\right)$
\item $\left( \frac{1}{7},\frac{1}{7}, \frac{3}{14}, \frac{3}{14}, \frac{2}{7} \right)$
\end{itemize}
To see that this lists exhausts all possibilities, write $\Lc^*(\Om)=\{a,b\}$ with $a<b$. If $a$ has multiplicity 3, then $\Lc(\Om)=\{a,a,a,b,b\}$.   By Proposition~\ref{prop:n_smoothbdd2}, $2b$ must be a sum of entries other than $b$.  The only possibility is that $3a=2b =\frac 1 2$, yielding the first item.   If $a$ has multiplicity 2, then again applying Proposition~\ref{prop:n_smoothbdd2}, we have $\Lc(\Om)=\{a, a, 2a, b, b\}$ (not necessarily in increasing order). We have $4a+2b =1$, so $2a+b =\frac 1 2$.  Moreover, $2b$ must be a sum of some of the entries not equal to $b$.   Since $b$ is not a multiple of $a$, the only possibility is that $2b=3a$, resulting in item 2.
\end{exa}

We next consider pentagons with an odd angle.
\begin{prop} \label{res:pentagons}
Assume that $\Omega$ is a pentagon of perimeter one with the same characteristic polynomial as a simply-connected smoothly bounded domain. Assume $\Om$ has an odd angle.  Choosing the cyclic labelling so that  $\alpha_1$ is odd, then one of the following assertions holds:
\[\begin{array}{cll}
 \ell_3=\ell_4=\ell_5=\frac{1}{4}& \mbox{ and }& \ell_1+\ell_2=\frac{1}{4},  \\ 
 \ell_3=\ell_5=\frac{1}{6}& \mbox{ and } &\ell_4=\ell_1+\ell_2=\frac{1}{3},\\
 \ell_3=\ell_5=\frac{1}{3}& \mbox{ and }& \ell_4=\ell_1+\ell_2=\frac{1}{6}.\\
\end{array}\]
\end{prop}

\begin{proof} Applying Propositions~\ref{prop:n_smoothbdd2} and \ref{prop:n_smoothbdd1}, along with Corollary~\ref{cor:not 1/2}, we find that the only possibilities for $\Lc(\omred)$ are $\left(\frac{1}{4},\frac{1}{4}, \frac{1}{4}, \frac{1}{4}\right)$ and $\left(\frac{1}{6},\frac{1}{6}, \frac{1}{3}, \frac{1}{3}\right)$.   To complete the proof, it suffices to show in the latter case that the edge lengths $\frac{1}{6}$ and $\frac{1}{3}$ must occur in alternating order in $\omred$.

Suppose to the contrary that the edges of equal length are adjacent.  For simplicity, we choose a new cyclic labelling, $\gamma_1, \dots, \gamma_4$, of the vertices of $\omred$ so that the two edges of length $\frac{1}{6}$ have endpoints $\gamma_1, \gamma_2$ and $\gamma_2,\gamma_3$, respectively.  Consequently, $\gamma_4$ is the common endpoint of the two edges of length $\frac{1}{3}$.   Recall that the characteristic polynomials of $\Om$ and of $\omred$ are identical except possibly for the sign of the constant term.  Thus comparing the coefficient of $\cos\left(\frac{2}{3}t\right)$ in $P_{\Om}$, equivalently in $P_{\omred}$, with that of a smoothly bounded simply-connected domain, we obtain 
\beq\label{eq:2/3 coef} c(\gamma_2)[c(\gamma_1)+c(\gamma_3)]=0.\eeq
Similarly comparing the coefficients of $\cos\left(\frac{1}{3}t\right)$, we obtain
\beq\label{eq:1/3 coef} c(\gamma_4)[c(\gamma_1)+c(\gamma_3)] + c(\gamma_1)c(\gamma_3)=0.\eeq
All of the $c(\gamma_j)$ are non-zero since $\omred$ has no odd angles.   Thus Equation~\eqref{eq:2/3 coef} implies that $c(\gamma_1)+c(\gamma_3)=0$.   But then Equation~\eqref{eq:1/3 coef} implies that one of $c(\gamma_1)$ or $c(\gamma_3)$ must vanish, yielding a contradiction.  This completes the proof.
\end{proof}

\begin{remark}\label{rem:equil triangle comparison}
Since the characteristic polynomial of an equilateral triangle agrees with that of a simply-connected smoothly bounded domain except for the sign of the constant term, the results in this subsection can be adapted to obtain necessary conditions for a convex $n$-gon, $n\geq 5$, of perimeter one to have the same characteristic polynomial as an equilateral triangle. First note that Proposition~\ref{prop:vertices} shows that the cosine terms in the characteristic polynomial distinguish all admissible curvilinear $n$-gons with $n\geq 2$ from the equilateral triangle.  Admissible 1-gons are distinguished from the equilateral triangle by their constant term.   Next, replacing ``simply-connected  smoothly bounded domain'' by ``equilateral triangle'' throughout, Lemma~\ref{lem:angle in +}, Proposition~\ref{prop:n_smoothbdd2}, Remark \ref{rem:shortest}, and Corollary \ref{cor:lc* size} all go through without change since they only rely on the cosine terms.   In the case of Proposition \ref{prop:n_smoothbdd1}, we need to replace ``opposite parity'' by ``same parity'' in part (i) and replace ``negative'' by ``positive'' in part (ii).   The proof then goes through with minor modifications, and then Corollary~\ref{cor:equilateral} goes through without change.   

Corollary~\ref{cor:not 1/2} also goes through with some changes in the proof as follows:    Define $T^*$, $\Om'$, $\beta_2$, $\beta_n$, $\alpha_2'$, $\alpha_n'$ and $s$ as before.  Since $\al_2$ and $\al_n$ must now be even angles of the same parity, we find that $\al_2=\al_n=\frac{\pi}{2}$ and thus, a priori, $\alpha_1$ can be any odd angle.   We have $\beta_2 + \beta_n \geq \frac{2\pi}{3}$ and thus $\alpha_2'+\alpha_n'\leq \pi-\frac{2\pi}{3} =\frac{\pi}{3}$.  A geometric argument shows that among all convex polygons $P$ with one edge of length $s$ and adjacent angles $\alpha_2'$ and $\alpha_n'$, the sum $S$ of the lengths of the remaining edges is maximized when $P$ is the triangle $T'$ determined by this data.   Moreover, $S$ is at most the sum of the two equal edge lengths of the isosceles triangle with base $s$ and adjacent angles $\gamma:=\frac{\alpha_2'+\alpha_n'}{2}$.  In our case $\gamma \leq \frac{\pi}{6}$, so we have 
$$\ell_3+\ell_4+\dots+\ell_n = S\leq \frac{s}{\cos(\pi/6)} = \frac{2s}{\sqrt{3}}<1.2 s.$$ The supremum $s^*$ of $s$ is the same as before, obtained when the odd angle $\alpha_1$ is maximal, i.e., $\alpha_1=\frac{\pi}{3}$ and $T$ is a right triangle; thus $s<0.3$, giving the desired contradiction $\ell_3+\ell_4+\dots+\ell_n<\frac{1}{2}$.

Example~\ref{exa:penta} and Proposition~\ref{res:pentagons} then go through without change.

\end{remark}

Based on the evidence throughout this section, we make the following bold conjecture:

\begin{conj}
A convex polygonal domain cannot be Steklov isospectral to a simply-connected smoothly bounded domain.
\end{conj}

\end{document}